\documentclass[11pt,a4paper]{article}
\title{Big jobs arrive early: \\ From critical queues to random graphs}
\date{\today}
\author{Gianmarco Bet, Remco van der Hofstad, and Johan S.H. van Leeuwaarden}

\newcommand{\FE}{F_{\sss E}}
\newcommand{\fE}{f_{\sss E}}

\newcommand{\served}[1]{\mathfrak S_{#1}}

\usepackage[]{xcolor}
\usepackage[]{todonotes} 
\usepackage{amsfonts}
\usepackage{amsmath}\numberwithin{equation}{section}
\usepackage{amsthm}
\usepackage{mathtools} 
\usepackage{geometry} 
\usepackage{enumitem} 
\usepackage{dsfont}
\usepackage{bm} 
\usepackage{soul} 
\usepackage{diagbox} 
\usepackage{standalone} 
\usepackage{pgfplots}

\theoremstyle{plain}
\newtheorem{lemma}{Lemma}
\newtheorem{theorem}{Theorem}
\newtheorem*{theorem*}{Theorem}
\newtheorem{corollary}{Corollary}

\theoremstyle{definition}

\newtheorem{remark}{Remark}

\theoremstyle{remark}

\newcommand{\Exp}{\mathrm{Exp}}
\newcommand{\sss}{\scriptscriptstyle}
\newcommand{\nin}{\notin}

\newcommand{\nnl}{\notag\\}
\newcommand{\sr}{\stackrel}
\newcommand{\E}{\mathbb E}

\newcommand{\OP}{O_{\mathbb P}}
\newcommand{\oP}{o_{\mathbb P}}
\newcommand{\DG}{\smash{\Delta_{(i)}/G/1}}

\newcommand{\dconv}{\sr{\mathrm d}{\rightarrow}}
\newcommand{\Pconv}{\sr{\mathbb P}{\rightarrow}}
\newcommand{\asconv}{\sr{\mathrm{a.s.}}{\rightarrow}}
\newcommand{\e}{\mathrm e}
\newcommand{\DGa}{\smash{\Delta\!%
\begin{smallmatrix}
\!\alpha \\
(i)
\end{smallmatrix}/G/1}}

\newcommand{\bl}[1]{#1}

\newcommand{\ES}{\E_{\sss S}}
\newcommand{\PS}{\mathbb P_{\sss S}}
\newcommand{\OPS}{O_{\PS}}

\begin{document}
\maketitle

\begin{abstract}
We consider a queue to which only a finite pool of $n$ customers can arrive, at times depending on their service requirement.  A customer with stochastic service requirement $S$ arrives to the queue after an exponentially distributed time with mean $S^{-\alpha}$ for some $\alpha\in[0,1]$; so larger service requirements trigger customers to join earlier. This finite-pool queue interpolates between two previously studied cases: $\alpha = 0$ gives the so-called $\DG$ queue \cite{bet2014heavy} and $\alpha = 1$ is closely related to the exploration process for inhomogeneous random graphs \cite{bhamidi2010scaling}.
We consider the asymptotic regime in which the pool size $n$ grows to infinity and establish that the scaled queue-length process converges  to a diffusion process with a negative quadratic drift. We leverage this asymptotic result to characterize the head start that is needed to create a long period of activity. We also describe how this first busy period of the queue gives rise to a critically connected random forest.
\end{abstract}

\section{Introduction} \label{sec:introduction}
This paper introduces the $\DGa$ queue that models a situation in which only a finite pool of $n$ customers will join the queue. These $n$ customers are triggered to join the queue after independent exponential times, but the rates of their exponential clocks depend on their service requirements. When a customer requires $S$ units of service, its exponential clock rings after an exponential time with mean $S^{-\alpha}$ with $\alpha\in[0,1]$. Depending on the value of the free parameter $\alpha$, the arrival times are i.i.d.~($\alpha = 0$) or decrease with the service requirement ($\alpha\in(0,1]$). The queue is attended by a single server that starts working at time zero, works at unit speed, and serves the customers in order of arrival. At time zero, we allow for the possibility that $i$ of the $n$ customers have already joined the queue, waiting for service. We will take $i\ll n$, so that without loss of generality we can assume that at time zero there are still $n$ customers waiting for service.  These initial customers are numbered $1, \ldots, i$ and the customers that arrive later are numbered $i + 1, i + 2,\ldots$ in order of arrival. Let $A(k)$ denote the number of customers arriving during the service time of the $k$-th customer. The busy periods of this queue will then be completely characterized by the initial number of customers $i$ and the random variables $\left(A(k)\right)_{k\geq 1}$. Note that the random variables $\left(A(k)\right)_{k\geq 1}$ are not i.i.d.~due to the finite-pool effect and the service-dependent arrival rates. We will model and analyze this queue using the queue-length process embedded at service completions.

We consider the $\DGa$ queue in the large-system limit $n\to\infty$, while imposing at the same time a heavy-traffic regime that will stimulate the occurrence of a substantial first busy period. By substantial we mean that the server can work without idling for quite a while, not only serving the initial customers but also those arriving somewhat later. For this regime we show that the embedded queue-length process converges to a Brownian motion with negative quadratic drift. For the case $\alpha = 0$, referred to as the  $\DG$ queue with i.i.d.~arrivals \cite{honnappa2015delta, honnappa2014transitory}, a similar regime was studied in \cite{bet2014heavy}, while for $\alpha = 1$ it is closely related to the critical inhomogeneous random graph studied in \cite{bhamidi2010scaling,Jose10}.



While the queueing process consists of alternating busy periods and idle periods, in the $\DGa$ queue we naturally focus on the first busy period. After some time, the activity in the queue inevitably becomes negligible. The early phases of the process are therefore of primary interest, when the head start provided by the initial customers still matters and when the rate of newly arriving customers is still relatively high. The head start and strong influx together lead to a substantial first busy period, and essentially determine the relevant time of operation of the system.

We also consider the structural properties of the first busy period in terms of a random graph.
Let the random variable $H(i)$ denote the number of customers served in the first busy period, starting with $i$ initial customers. We then  associate a (directed) random graph to the queueing process as follows. Say $H(i) = N$ and consider a graph with vertex set $\{1, 2,\ldots,N\}$ and in which two vertices $r$ and $s$ are joined by an edge if and only if the $r$-th customer arrives during the service time of the $s$-th customer. If $i = 1$, then the graph is a rooted tree with $N$ labeled vertices, the root being labeled 1. If $i > 1$, then the graph is a forest consisting of $i$ distinct rooted trees whose roots are labeled $1,\ldots,i$, respectively. The total number of vertices in the forest is $N$.

This random forest is exemplary for a deep relation between queues and random graphs, perhaps best explained by
interpreting the embedded  $\DGa$ queue as an \emph{exploration process}, a generalization of a branching process that can account for dependent random variables $\left(A(k)\right)_{k\geq 1}$. Exploration processes arose in the context of random graphs as a recursive algorithm to investigate questions concerning the size and structure of the largest components \cite{Aldo97}. For a given random graph, the exploration process declares vertices active, neutral or inactive. Initially, only one vertex is active and all others are neutral. At each time step one active vertex (e.g. the one with the smallest index) is explored, and it is declared inactive afterwards. When one vertex is explored, its neutral neighbors become active for the next time step. As time progresses, and more vertices are already explored (inactive) or discovered (active), fewer vertices are neutral. This phenomenon is known as the \emph{depletion-of-points effect} and plays an important role in the scaling limit of the random graph. Let $A(k)$ denote the neutral neighbors of the $k$-th explored vertex. The exploration process then has increments $\left(A(k)\right)_{k\geq 1}$ that each have a different distribution. The exploration process encodes useful information about the underlying random graph. For example, excursions above past minima are the sizes of the connected components. The critical behavior of random graphs connected with the emergence of a giant component has received tremendous attention \cite{ addario2012continuum,bhamidi2014augumented,bhamidi2010scaling,BhaHofLee09b,
dhara2016critical, dhara2016heavy, Jose10, hofstad2010critical, hofstad2016mesoscopic}. Interpreting active vertices as being in a queue, and vertices being explored as customers being served, we see that the exploration process and the (embedded) $\DGa$ queue driven by $\left(A(k)\right)_{k\geq 1}$ are identical.

The analysis of the $\DGa$ queue and associated random forest is challenging because the random variables $\left(A(k)\right)_{k\geq 1}$ are not i.i.d. In the case of i.i.d.~$\left(A(k)\right)_{k\geq 1}$,  there exists an even deeper connection between queues and random graphs, established via branching processes instead of exploration processes \cite{kendall1951some}. To see this, declare the initial customers in the queue to be the $0$-th generation. The customers (if any) arriving during the total service time of the initial $i$ customers form the $1$-st generation, and the customers (if any) arriving during the total service time of the customers in generation $t$ form generation $t+1$ for $t \geq 1$. Note that the total progeny of this Galton-Watson branching process has the same distribution as the random variable $H(i)$ in the queueing process. Through this connection, properties of branching processes can be carried over to the queueing processes and associated random graphs  \cite{duquesne2002random, legall2005random,limic2001lifo,takacs1988queues,takacs1993limit,takacs1995queueing}. Tak\'acs \cite{takacs1988queues,takacs1993limit,takacs1995queueing} proved several limit theorems for the case of i.i.d.~$\left(A(k)\right)_{k\geq 1}$, in which case the queue-length process and derivatives such as the first busy period weakly converge to (functionals of) the Brownian excursion process. In that classical line, the present paper can be viewed as an extension to exploration processes with more complicated dependency structures in $\left(A(k)\right)_{k\geq 1}$.


In Section \ref{sec:model_description} we describe the $\DGa$ queue and associated graphs in more detail and present our main results. The proof of the main theorem, the stochastic-process limit for the queue-length process in the large-pool heavy-traffic regime, is presented in Sections \ref{sec:overview_proof} and \ref{sec:proof_main_result}. Section \ref{sec:conclusion} discusses some interesting questions related to the $\DGa$ queue and associated random graphs that are left open.

\section{Model description} \label{sec:model_description}
We consider a sequence of queueing systems, each with a finite (but growing) number $n$ of potential customers labelled with indices $i\in[n]:=\{1,\ldots, n\}$. Customers have i.i.d.~service requirements with distribution $F_{\sss S}(\cdot)$. We denote with $S_i$ the service requirement of customer $i$ and with $S$ a generic random value, and $S_i$ and $S$ all have distribution $F_{\sss S}(\cdot)$. In order to obtain meaningful limits as the system grows large, we scale the service speed  by $n/(1+\beta n^{-1/3})$ with $\beta\in\mathbb R$ so that the service time of customer $i$ is given by
\begin{equation}
\tilde S_i = \frac{S_i(1+\beta n^{-1/3})}{n}.
\end{equation}%
We further assume that $\E[{S}^{2+\alpha}]<\infty$.

If the service requirement of customer $i$ is $S_i$, then, conditioned on $S_i$, its arrival time $T_i$ is assumed to be exponentially distributed with mean $1/(\lambda S_i^{\alpha})$, with $\alpha\in[0,1]$ and $\lambda>0$. Hence
\begin{equation}\label{eq:arrival_times}%
T_i\sr{\mathrm{d}}{=}\Exp_i(\lambda S_i^{\alpha})
\end{equation}%
with $\sr{\mathrm{d}}{=}$ denoting equality in distribution and $\Exp_i(c)$  an exponential random variable with mean $1/c$ independent across $i$. Note that {conditionally on the service times}, the arrival times are independent (but not identically distributed). We introduce $c(1), c(2), \ldots, c(n)$ as the indices of the customers in order of arrival, so that $T_{c(1)} \leq T_{c(2)}\leq T_{c(3)}\leq\ldots$ almost surely.

We will study the queueing system in heavy traffic, in a similar heavy-traffic regime as in \cite{bet2014heavy, bet2015finite}. The initial traffic intensity $\rho_n$ is kept close to one by imposing the relation
\begin{equation}\label{eq:crit}%
\rho_n := \lambda_n \mathbb{E}[S^{1+\alpha}](1+\beta n^{-1/3}) = 1+\beta n^{-1/3} + \oP(n^{-1/3}),
\end{equation}%
where $\lambda = \lambda_n$ can depend on $n$ and $f_n = \oP(n^{-1/3})$ is such that $\lim_{n\rightarrow\infty} f_nn^{1/3}\Pconv 0$.  The parameter $\beta$ then determines the position of the system inside the critical window: the traffic intensity is greater than one for $\beta>0$, so that the system is initially overloaded, while the system is initially underloaded for $\beta<0$.

Our main object of study is the queue-length process embedded at service completions, given by $Q_n(0) = i$ and
\begin{align}\label{eq:queue_length_process}%
Q_n(k) &= (Q_n(k-1) + A_n(k)-1)^+,
\end{align}%
with $x^+ = \max\{0,x\}$ and $A_n(k)$ the number of arrivals during the $k$-th service given by
\begin{equation}\label{eq:number_arrivals_during_one_service}
A_n(k) = \sum_{i\nin\nu_{k}}\mathds{1}_{\{T_{i}\leq \tilde S _{c(k)}\}}
\end{equation}
where $\nu_{k}\subseteq[n]$ denotes the set of customers who have been served or are in the queue at the start of the $k$-th service. Note that
\begin{equation}\label{eq:cardinality_nu}%
\vert \nu_{k} \vert = (k-1) + Q_n(k-1) + 1 = k + Q_n(k-1).
\end{equation}%
%
Given a process $t\mapsto X(t)$, we define its \emph{reflected version} through the reflection map $\phi(\cdot)$ as
\begin{equation}%
\phi(X)(t) := X(t) - \inf_{s\leq t} X(s)^-.
\end{equation}%
%
The process $Q_n(\cdot)$ can alternatively be represented as the reflected version of a certain process $N_n(\cdot)$, that is
\begin{equation}\label{eq:queue_length_process_reflected}%
Q_n(k) = \phi (N_n)(k),
\end{equation}%
where $N_n(\cdot)$ is given by $N_n(0)=i$ and
\begin{align}\label{eq:pre_reflection_queue_length_process}%
N_n(k)&=N_n(k-1)+A_n(k)-1.
\end{align}%
We assume that whenever the server finishes processing one customer, and the queue is empty, 
the customer to be placed into service is chosen according to the following size-biased distribution:
\begin{align}\label{eq:choice_of_customer_when_queue_is_empty}%
\mathbb P (\mathrm{customer}~j~\mathrm{is~placed~in~service}\mid \nu_{i-1}) = \frac{S_j^{\alpha}}{\sum_{l\nin \nu_{i-1}}S_l^{\alpha}},\qquad j\nin\nu_{i-1},
\end{align}%
where we tacitly assumed that customer $j$ is the $i$-th customer to be served. \bl{With definitions \eqref{eq:number_arrivals_during_one_service} and \eqref{eq:choice_of_customer_when_queue_is_empty}, the process \eqref{eq:queue_length_process} describes the $\DGa$ queue with exponential arrivals \eqref{eq:arrival_times}, embedded at service completions. }
\begin{remark}[A directed random tree] \label{rem:directed_random_tree_model} The embedded queueing process \eqref{eq:queue_length_process} and \eqref{eq:queue_length_process_reflected} gives rise to a certain directed rooted tree. To see this, associate a vertex $i$ to customer $i$ and let $c(1)$ be the root. Then, draw a directed edge to $c(1)$ from $c(2),\ldots, c(A_n(1)+1)$ so to all customers who joined during the service time of $c(1)$. Then, draw an edge from all customers who joined during the service time of $c(2)$ to $c(2)$, and so on. This procedure draws a directed edge from $c(i)$ to $c(i + \sum_{j=1}^{i-1}A_n(j)),\ldots,c(i + \sum_{j=1}^{i}A_n(j) )$ if $A_n(i) \geq 1$. The procedure stops when the queue is empty and there are no more customers to serve. When $Q_n(0) = i = 1$ (resp.~$i\geq 2$), this gives a random directed rooted tree (resp.~forest). The degree of vertex $c(i)$ is $1 + \vert A_n(i)\vert$ and the total number of vertices in the tree (resp.~forest) is given by
\begin{equation}
H_{Q_n}(0) = \inf\{k\geq0:Q_n(k)=0\},
\end{equation}
the hitting time of zero of the process $Q_n(\cdot)$.
\end{remark}
\begin{remark}[An inhomogeneous random graph] If $\alpha = 1$, the random tree constructed in Remark \ref{rem:directed_random_tree_model} is distributionally equivalent to the tree spanned by the exploration process of an inhomogeneous random graph. Let us elaborate on this. An inhomogeneous random graph is a set of vertices $\{i:i\in[n]\}$ with (possibly random) weights $(\mathcal W_i)_{i\in[n]}$ and edges between them. In a \emph{rank-1 inhomogeneous random graph}, given $(\mathcal W_i)_{i\in[n]}$, $i$ and $j$ share an edge with probability
\begin{equation}%
p_{i\leftrightarrow j} := 1- \exp\Big(-\frac{\mathcal W_i\mathcal W_j}{\sum_{i\in[n]}\mathcal W_i}\Big).
\end{equation}%

The tree constructed from the $\smash{\Delta\!%
\begin{smallmatrix}
\!1 \\
(i)
\end{smallmatrix}/G/1}$ queue then corresponds to the exploration process of a rank-1 inhomogeneous random graph, defined as follows. Start with a first arbitrary vertex and reveal all its neighbors. Then the first vertex is discarded and the process moves to a neighbor of the first vertex, and reveals its neighbors. This process continues by exploring the neighbors of each revealed vertex, in order of appearance.  By interpreting each vertex as a different customer, this exploration process can be coupled to a  $\smash{\Delta\!%
\begin{smallmatrix}
\!1 \\
(i)
\end{smallmatrix}/G/1}$ queue, for a specific choice of $(\mathcal W_i)_{i=1}^n$ and $\lambda_n$. Indeed, when $\mathcal W_i = (1+\beta n^{-1/3})S_i$ for $i=1,\ldots, n$, the probability that $i$ and $j$ are connected is given by
\begin{align}%
p_{j\leftrightarrow i} &= 1 - \exp\Big( -(1+\beta n^{-1/3})\frac{S_i}{n}\frac{ S_j}{\sum_{l\in[n]}S_l/n}\Big) \nnl
&= 1 - \exp\Big( -\tilde{S}_iS_j\frac{n}{\sum_{i\in[n]}S_i}\Big)\nnl
&=\mathbb P(T_j\leq \tilde S_i\vert (S_i)_{j\in[n]}),
\end{align}%
where
\begin{equation}%
T_j \sim \exp (\lambda_n),
\end{equation}%
and $\lambda_n = {n}/\sum_{i\in[n]} S_i$.
The rank-1 inhomogeneous random graph  with weights $(S_i)_{i=1}^n$ is said to be \emph{critical} (see \cite[(1.13)]{bhamidi2010scaling}) if
\begin{equation}\label{eq:crit_inhomogeneous}%
\frac{\sum_{i\in[n]}S_i^2}{\sum_{i\in[n]}S_i}=\frac{\E[S^2]}{\E[S]} + \oP(n^{-1/3})= 1 + \oP(n^{-1/3}).
\end{equation}%
Consequently, if $\beta=0$ and $\lambda_n = n/\sum_{i\in[n]}S_i$, the heavy-traffic condition \eqref{eq:crit} for the $\smash{\Delta\!%
\begin{smallmatrix}
\!1 \\
(i)
\end{smallmatrix}/G/1}$ queue implies  the criticality condition \eqref{eq:crit_inhomogeneous} for the associated random graph (and vice versa).
\end{remark}
\begin{remark}[Results for the queue-length process]\label{rem:results_queue_length_process}
\bl{By definition, the embedded  queue \eqref{eq:queue_length_process} neglects the idle time of the server.} Via a time-change argument it is possible to prove that, in the limit, the (cumulative) idle time is negligible and the embedded queue is arbitrarily close to the queue-length process uniformly over compact intervals. This has been proven for the $\DG$ queue in \cite{bet2014heavy}, and  the techniques developed there can be extended to the $\DGa$ queue without additional difficulties. 
\end{remark}

\subsection{The scaling limit of the embedded queue}

All the processes we consider are elements of the space $\mathcal D := \mathcal D([0,\infty))$ of c\`adl\`ag functions that admit left limits and are continuous from the right.  To simplify notation, for a discrete-time process $X(\cdot):\mathbb N \rightarrow\mathbb R$, we write $X(t)$, with $t\in[0,\infty)$, instead of $X(\lfloor t\rfloor)$. Note that a process defined in this way has c\`adl\`ag paths. The space $\mathcal D$ is endowed with the usual Skorokhod $J_1$ topology. We then say that a process converges in distribution in $(\mathcal D, J_1)$ when it converges as a random measure on the space $\mathcal D$, when this is endowed with the $J_1$ topology.
{We are now able to state our main result. Recall that $Q_n(\cdot)$ is the embedded queue-length process of the $\DGa$ queue and let
\begin{equation}\label{def:rescaled_Q}
\mathbf{Q}_n(t) := n^{-1/3} Q_n(tn^{2/3})
\end{equation}%
be the diffusion-scaled queue-length process.

\begin{theorem}[Scaling limit for the $\DGa$ queue]\label{th:main_theorem_delta_G_1}
Assume that $\alpha\in[0,1]$, $\E[S^{2+\alpha}]<\infty$ and that the heavy-traffic condition \eqref{eq:crit} holds. Assume further that $\mathbf{Q}_n(0) = q $. Then, as $n\rightarrow\infty$,
\begin{equation}\label{eq:main_theorem_delta_G_1_conclusion}%
\mathbf Q_n( \cdot )\stackrel{\mathrm{d}}{\rightarrow} \phi(W)( \cdot )\qquad \mathrm{in}~(\mathcal D, J_1),
\end{equation}%
where $W(\cdot)$ is the diffusion process
\begin{equation}\label{eq:main_theorem_delta_G_1_diffusion_definition}%
W(t)= q + \beta t - \lambda\frac{\E[S^{1+2\alpha}]}{2\E[S^{\alpha}]}t^2 + \sigma B(t),
\end{equation}%
with $\sigma^2 = \lambda^2 \E[S^{\alpha}]\mathbb E[S^{2+\alpha}]$ and $B(\cdot)$ is a standard Brownian motion.
\end{theorem}

By the Continuous-Mapping Theorem and Theorem \ref{th:first_busy_period_convergence} we have the following:
\begin{theorem}[Number of customers served in the first busy period]\label{th:first_busy_period_convergence}%
Assume that $\alpha\in[0,1]$, $\E[S^{2+\alpha}]<\infty$ and that the heavy-traffic condition \eqref{eq:crit} holds. Assume further that $\mathbf Q_n(0) = q $. Then, as $n\rightarrow\infty$, the number of customers served in the first busy period $\mathrm{BP}_n:=H_{\mathbf{Q}_n}(0)$ converges to
\begin{equation}%
\mathrm{BP}_n \dconv H_{\phi(W)}(0),
\end{equation}%
where $W(\cdot)$ is given in \eqref{eq:main_theorem_delta_G_1_diffusion_definition}.
\end{theorem}%
}
In particular, if $\vert \mathcal F_n \vert$ denotes the number of vertices in the forest constructed from the $\DGa$ queue in Remark \ref{rem:directed_random_tree_model}, as $n\rightarrow\infty$,
\begin{equation}%
\vert \mathcal F_n \vert \dconv H_{\phi(W)}(0).
\end{equation}%
{

Theorem \ref{th:main_theorem_delta_G_1} implies that the typical queue length for the $\DGa$ system in heavy traffic is $\OP(n^{1/3})$, and that the typical busy period consists of $\OP(n^{2/3})$ services. The linear drift $t\rightarrow\beta \lambda t$ describes the position of the system inside the critical window. For $\beta>0$ the system is initially overloaded and the process $W(\cdot)$ is more likely to cause a large initial excursion. For $\beta<0$ the traffic intensity approaches $1$ from below, so that the system is initially stable. Consequently, the process $W(\cdot)$ has a strong initial negative drift, so that $\phi(W)(\cdot)$ is close to zero also for small $t$. Finally, the negative quadratic drift $t\rightarrow - \lambda\frac{\E[S^{1+2\alpha}]}{2\E[S^{\alpha}]} t^2$ captures the \emph{depletion-of-points effect}. Indeed, for large times, the process $W(t)$ is dominated by $- \lambda\frac{\E[S^{1+2\alpha}]}{2\E[S^{\alpha}]} t^2$, so that $\phi(W)(t)$ performs only small excursions away from zero.  See Figure \ref{fig:stable_motion_different_linear_drift_examples}.
\begin{figure}[hbt]
\centering
\includestandalone[mode=image]{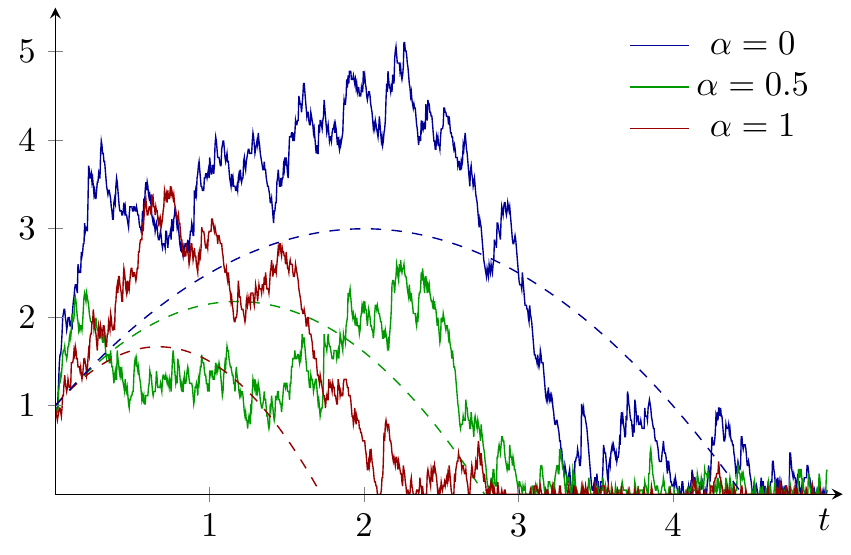}
\caption{Sample paths of the process $\mathbf Q_n(\cdot)$  for various values of $\alpha$ and $n= 10^4$. The service times are taken unit-mean exponential. The dashed curves represent the drift  $t\mapsto q+  \beta t - \lambda \E[S^{1+2\alpha}]/(2\E[S^{\alpha}]) t^2$. In all plots, $q = 1$, $\beta=1$, $\lambda=1/\E[S^{1+\alpha}]$.}	
\label{fig:stable_motion_different_linear_drift_examples}
\end{figure}

Let us now compare Theorem \ref{th:main_theorem_delta_G_1} with two known results.
For $\alpha = 0$, the limit diffusion simplifies to
\begin{equation}\label{eq:limit_diffusion_alpha_0}
W(t) = \beta t- \frac{1}{2} t^2 + \sigma B(t),
\end{equation}
with $\sigma ^2 = \lambda^2 \E[S^2]$, in agreement with \cite[Theorem 5]{bet2014heavy}. In \cite{bhamidi2010scaling} it  is shown that, when $(\mathcal W_i)_{i\in[n]}$ are i.i.d.~and further assuming that $\E[\mathcal W^2]/\E[\mathcal W] = 1$, the exploration process of the corresponding inhomogeneous random graph converges to
\begin{align}%
\overline{W}(t) &= \beta t - \frac{\E[\mathcal W^3]}{2\E[\mathcal W^2]^2}t^2 + {\frac{\sqrt{\E[\mathcal W]\E[\mathcal W^3]}}{\E[\mathcal W^2]}} B(t).
\end{align}%
For $\alpha = 1$, \eqref{eq:main_theorem_delta_G_1_diffusion_definition} can be rewritten using \eqref{eq:crit} as
\begin{equation}\label{eq:limit_diffusion_alpha_1}%
W(t) = \beta t - \frac{\E[S^3]}{2\E[S^2]^2} t^2 + {\frac{\sqrt{\E[S]\E[S^3]}}{\E[S^2]}}B(t).
\end{equation}%
Therefore the two processes coincide if $\mathcal W_i = S_i$, as expected.
\subsection{Numerical results}
We now use Theorem \ref{th:first_busy_period_convergence} to obtain numerical results for the first busy period. We shall also use the explicit expression of the probability density function of the first passage time of zero of $\phi(W)$ obtained by Martin-L\"of \cite{martin1998final}, see also \cite{hofstad2010critical}. Let Ai$(x)$ and Bi$(x)$ denote the classical Airy functions (see \cite{abramowitz1964handbook}). The first passage time of zero of $W(t) = q + \beta t -1/2t^2 + \sigma B(t)$ has probability density \cite{martin1998final}
\begin{equation}\label{eq:first_busy_period_density}%
f(t;\beta,\sigma) = \e^{-((t-\beta)^3+\beta^3)/6\sigma^2-\beta a}\int_{-\infty}^{+\infty}\e^{tu}\frac{\mathrm{Bi}(cu)\mathrm{Ai}(c(u-a))-\mathrm{Ai}(cu)\mathrm{Bi}(c(u-a))}{\pi(\mathrm{Ai}(cu)^2 + \mathrm{Bi}(cu)^2)}\mathrm{d} u,
\end{equation}%
where $c = (2\sigma^2)^{1/3}$ and $a=q/\sigma^2>0$. The result \eqref{eq:first_busy_period_density} can be extended to a diffusion with a general quadratic drift through the scaling relation $W(\tau^2 t) = \tau(q/\tau + \beta \tau t - \tau^3 t^2/2 + \sigma B(t))$.
\begin{figure}[!hbt]
	\centering
	\includestandalone[mode=image|tex]{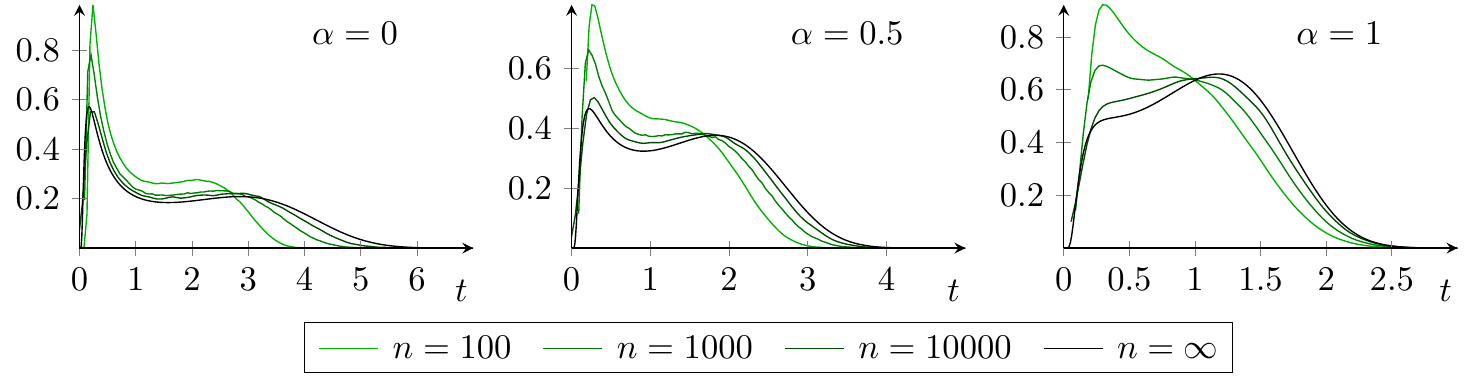}
	\caption{Density plot (black) and Gaussian kernel density estimates (colored) obtained by running $10^6$ simulations of a $\Delta_{(i)}^{\alpha}/G/1$ queue with $n=\,100,\,1000,\,10000$ customers and $\alpha = 0,~1/2,~1$. In all cases, the service times are exponentially distributed and $q=\beta=\E[S]=1$.}			
	\label{fig:distribution_first_busy_period}
\end{figure}
Figure \ref{fig:distribution_first_busy_period} shows the empirical density of $\mathrm{BP}_n$, for increasing values of $n$ and various values of $\alpha$, together with the exact limiting value \eqref{eq:first_busy_period_density}. 

\begin{table}[!htbp]
\centering
\begin{tabular}{ c  c c c c c c c c c}
 & \multicolumn{3}{ c }{Deterministic} & \multicolumn{3}{ c }{Exponential}& \multicolumn{3}{ c }{Hyperexponential}\\
%
$\alpha$ & 0 & 1/2 & 1 & 0 & 1/2 & 1 & 0 & 1/2 & 1\\
%
\hline
$n$ \\
\multicolumn{1}{r}{$10^1$} & 1.1318	& 1.1318	& 1.1318	& 1.0359 	& 0.8980 & 0.7429	& 0.8920 &	0.6356 & 0.5332	\\
\multicolumn{1}{r}{$10^2$} & 1.5842	& 	1.5842 & 1.5842	& 1.3584	& 1.0924 & 0.8333	& 1.0959	&	0.7454 & 0.5525 \\
\multicolumn{1}{r}{$10^3$} &	1.9188	& 	1.9188 & 1.9188	& 1.6387	& 1.2506	& 0.9284	& 1.2936	&	0.8352 &	0.6134	\\
\multicolumn{1}{r}{$10^4$} &	2.1474	& 	2.1474 & 2.1474	&	1.8419 & 1.3925	& 	1.0014 & 1.4960	&	0.9210 &	0.6554	\\
%
%
\multicolumn{1}{c}{$\infty$~~} &	2.3374	& 	2.3374    & 2.3374 &	2.0038 & 1.4719	& 1.0440	& 1.6242	&	0.9717 &	0.6881
\end{tabular}
\caption{Numerical values of $n^{-2/3}\E[\mathrm{BP}_n]$ for different population sizes and the exact expression for $n=\infty$ computed using \eqref{eq:first_busy_period_density}. The service requirements are displayed in order of increasing coefficient of variation. In all cases $q = \beta = \E[S] = 1$. The hyperexponential service times follow a rate $\lambda_1 = 0.501$ exponential distribution with probability $p_1=1/2$ and a rate $\lambda_2 = 250.5$ exponential distribution with probability $p_2=1-p_1=1/2$. Each value for finite $n$ is the average of $10^4$ simulations. }\label{tab:convergence_expectation_busy_period}
\end{table}
Table \ref{tab:convergence_expectation_busy_period} shows the mean busy period for different choices of $\alpha$ and different service time distributions. We computed the exact value for $n=\infty$ by numerically integrating \eqref{eq:first_busy_period_density}.
Observe that $\E[\mathrm{BP}_n]$ decreases with $\alpha$. This might seem counterintuitive, because the larger $\alpha$, the more likely customers with larger service join the queue early, who in turn might initiate a large busy period. Let us explain this apparent contradiction. When the arrival rate $\lambda$ is fixed, assumption \eqref{eq:crit} does not necessarily hold and $\E[\mathrm{BP}_n]$ increases with $\alpha$, as can be seen in Table \ref{tab:fixed_lambda_mean_busy_period}.
\begin{table}[!b]%
\centering
\begin{tabular}{c c c c c c}
&\multicolumn{5}{c}{ Exponential }\\
$\alpha$ & 0 & 1/4 & 1/2 & 3/4 & 1 \\
\hline
$n$ \\
$10^1$ & 1.0854 & 1.0922 & 1.1053 & 1.1118 & 1.1306 \\
$10^2$ & 5.9515 & 8.1928 & 11.4478 & 16.3598 & 22.0381 \\
\end{tabular}
\caption{Expected number of customers served in the first busy period of the \emph{nonscaled} $\Delta_{(i)}^{\alpha}/G/1$ queue with mean one exponential service times and arrival rate $\lambda = 0.01$. In all cases $q=1$. Each value is the average of $10^4$ simulations.}\label{tab:fixed_lambda_mean_busy_period}
\end{table}%
However, our heavy-traffic condition \eqref{eq:crit} implies that $\lambda$ depends on $\alpha$ since $\lambda = 1/\E[S^{1+\alpha}]$. The interpretation of condition \eqref{eq:crit} is that, on average, one customer joins the queue during one service time. Notice that, due to the size-biasing, the average service time is not $\E[S]$. Therefore, the number of customers that join during a (long) service is roughly equal to one as $\alpha\uparrow 1$. However, when customers with large services leave the system, they are not able to join any more. As $\alpha\uparrow 1$, customers with large services leave the system earlier. Therefore, as $\alpha\uparrow1$, the resulting second order \emph{depletion-of-points effect} causes shorter excursions as time progresses, see also Figure \ref{fig:stable_motion_different_linear_drift_examples}. In the limit process, this phenomenon is represented by the fact that the coefficient of the negative quadratic drift increases as $\alpha\uparrow1$, as shown in the following lemma.
{
\begin{lemma}%
Let
\begin{equation}\label{eq:coefficient_negative_quadratic_drift}%
\alpha\mapsto f(\alpha) := \frac{\E[S^{1+2\alpha}]}{\E[S^{\alpha}]\E[S^{1+\alpha}]}.
\end{equation}%
Then $f'(\alpha)\geq0$.
\end{lemma}%
}
\begin{proof}%
Since
\begin{equation}%
f'(\alpha) = \frac{2\E[\log(S)S^{1+2\alpha}]}{\E[S^{\alpha}]\E[S^{1+\alpha}]} - \frac{\E[S^{1+2\alpha}]\E[\log(S)S^{\alpha}]}{\E[S^{\alpha}]^2\E[S^{1+\alpha}]} - \frac{\E[S^{1+2\alpha}]\E[\log(S)S^{1+\alpha}]}{\E[S^{\alpha}]\E[S^{1+\alpha}]^2},
\end{equation}%
$f'(\alpha)\geq 0$ if and only if
\begin{align}%
2 \E[\log(S)S^{1+2\alpha}]\E[S^{\alpha}]\E[S^{1+\alpha}] &\geq\E[S^{1+\alpha}]\E[S^{1+2\alpha}]\E[\log(S)S^{\alpha}] \nnl
&\quad+ \E[S^{\alpha}]\E[S^{1+2\alpha}]\E[\log(S)S^{1+\alpha}].
\end{align}%
We split the left-hand side in two identical terms and show that each of them dominates one term on the right-hand side. That is
\begin{equation}\label{eq:coefficient_negative_drift_proof_first_term}%
\E[\log(S)S^{1+2\alpha}]\E[S^{\alpha}]\E[S^{1+\alpha}]  \geq \E[S^{1+\alpha}]\E[S^{1+2\alpha}]\E[\log(S)S^{\alpha}],
\end{equation}%
the proof of the second bound being analogous. The inequality \eqref{eq:coefficient_negative_drift_proof_first_term} is equivalent to
\begin{equation}\label{eq:coefficient_negative_drift_proof_first_term_rewritten}%
\frac{\E[(\log(S)S^{1+\alpha})S^{\alpha}]}{\E[S^{\alpha}]} \geq \frac{\E[S^{1+\alpha}S^{\alpha}]}{\E[S^{\alpha}]}\frac{\E[\log(S)S^{\alpha}]}{\E[S^{\alpha}]}.
\end{equation}%
The term on the left and the two terms on the right can be rewritten as the expectation of a size-biased random variable $W$, so that \eqref{eq:coefficient_negative_drift_proof_first_term_rewritten} is equivalent to
\begin{equation}\label{eq:coefficient_negative_drift_proof_final}%
\E[\log(W)W^{1+\alpha}] \geq \E[\log(W)]\E[W^{1+\alpha}] .
\end{equation}%
Finally, the inequality \eqref{eq:coefficient_negative_drift_proof_final} holds because $W$ is positive with probability one and $x\mapsto\log(x)$ and $x\mapsto x^{1+\alpha}$  are increasing functions.
\end{proof}%

\section{Overview of the proof of the scaling limit}\label{sec:overview_proof}
The proof of Theorem \ref{th:main_theorem_delta_G_1} extends the techniques we developed in \cite{bet2014heavy}.  However, the dependency structure of the arrival times complicates the analysis considerably. Customers with larger job sizes have a higher probability of joining the queue quickly, and this gives rise to a
size-biasing reordering of the service times. In the next section we study this phenomenon in detail.

\subsection{Preliminaries} \label{sec:preliminaries}

Given two sequences of random variables $(X_n)_{n\geq1}$ and $(Y_n)_{n\geq1}$, we say that $X_n$ converges in probability to $X$, and we denote it by $X_n\Pconv X$, if $\mathbb P (\vert X_n - X\vert >\varepsilon) \rightarrow0$  as $n\rightarrow0$ for each $\varepsilon>0$. We also write $X_n = \oP(Y_n)$ if $X_n/Y_n \sr{\mathbb P}{\rightarrow}0$ and $X_n=\OP(Y_n)$ if $(X_n/Y_n)_{n\geq1}$ is tight. Given two real-valued random variables $X$, $Y$ we say that $X$ \emph{stochastically dominates} $Y$ and denote it by $Y\preceq X$, if $\mathbb{P}(X\leq x) \leq \mathbb P(Y\leq x)$ for all $x\in\mathbb R$.

For our results, we condition on the entire sequence $(S_i)_{i\geq1}$. More precisely, if the random variables that we consider are defined on the probability space $(\Omega, {\mathcal F}, {\mathbb P})$, then we define a new probability space $(\Omega, \mathcal F_{\sss S}, \PS)$, with $ \PS(A) := {\mathbb P}(A\vert(S_i)_{i=1}^{\infty})$ and $\mathcal F_{\sss S} := \sigma(\{{\mathcal F} , (S_i)_{i=1}^{\infty}\})$, the $\sigma$-algebra generated by $\mathcal F$ and $(S_i)_{i=1}^{\infty}$. Correspondingly, for any random variable $X$ on $\Omega$ we define $\ES[X]$ as the expectation with respect to $\PS$, and $\E[X]$ for the expectation with respect to $\mathbb P$. We say that a sequence of events $(\mathcal E_n)_{n\geq1}$ holds with high probability (briefly, w.h.p.) if $\mathbb P (\mathcal E_n) \rightarrow1$ as $n\rightarrow\infty$.

First, we recall a well-known result that will be useful on several occasions.
\begin{lemma}\label{lem:max_iid_random_variables}
Assume $(X_i)_{i=1}^n$ is a sequence of positive i.i.d.~random variables such that $\E[X_i] <\infty$. Then $\max_{i\in[n]}X_i = \oP(n)$.
\end{lemma}%
\begin{proof}%
We have the inclusion of events
\begin{equation}
\Big\{\max_{i\in[n]} X_i \geq \varepsilon n \Big\} \subseteq \bigcup_{i=1}^n\Big\{X_i\geq \varepsilon n\Big\}.
\end{equation}%
Therefore,
\begin{equation}
\mathbb P (\max_{i\in[n]} X_i \geq \varepsilon n) \leq \sum_{i=1}^n \mathbb P (X_i \geq \varepsilon n).
\end{equation}%
Since for any positive random variable $Y$, $\varepsilon \mathds 1_{\{Y\geq\varepsilon\}}\leq Y \mathds 1_{\{Y\geq\varepsilon\}} $ almost surely, it follows
\begin{equation}%
\mathbb P (\max_{i\in[n]} X_i \geq \varepsilon n) \leq \frac{\sum_{i=1}^{n}\E[X_i\mathds 1_{\{X_i\geq\varepsilon n\}}]}{\varepsilon n} = \frac{\E[X_1 \mathds 1_{\{X_1\geq\varepsilon n\}}]}{\varepsilon}.
\end{equation}%
The right-most term tends to zero as $n\rightarrow\infty$ since $\E[X_1]<\infty$, and this concludes the proof.
\end{proof}%
%
%
Given a vector $\bar x =(x_1,x_2,\ldots,x_n)$ with deterministic, real-valued entries, the size-biased ordering of $\bar x$ is a \emph{random} vector $X\!\begin{smallmatrix}(s)\\
~\end{smallmatrix} = (X\!\begin{smallmatrix}(s)\\
\mkern-12mu1\end{smallmatrix},X\!\begin{smallmatrix}(s)\\
\mkern-12mu2\end{smallmatrix},\ldots,X\!\begin{smallmatrix}(s)\\
\mkern-12mu n\end{smallmatrix})$ such that
\begin{equation}%
\mathbb P (X\!\begin{smallmatrix}(s)\\
\mkern-12mu1\end{smallmatrix} = x_j) = \frac{x_j}{\sum_{l=1}^n x_l},~  \mathbb P (x\!\begin{smallmatrix}(s)\\
\mkern-12mu2\end{smallmatrix} = x_j \mid X\!\begin{smallmatrix}(s)\\
\mkern-12mu1\end{smallmatrix}) = \frac{x_j}{\sum_{l=1}^n x_l - x\!\begin{smallmatrix}(s)\\
\mkern-12mu1\end{smallmatrix}},~\ldots
\end{equation}%
More generally, for any $\alpha\in\mathbb R$ the $\alpha$-size-biased ordering of $\bar x$ is given by a vector $\bar X\!\begin{smallmatrix}(\alpha)\\
~\end{smallmatrix} = (X\!\begin{smallmatrix}(\alpha)\\
\mkern-12mu1\end{smallmatrix}, X\!\begin{smallmatrix}(\alpha)\\
\mkern-12mu2\end{smallmatrix},\ldots, X\!\begin{smallmatrix}(\alpha)\\
\mkern-12mu~n\end{smallmatrix})$ such that
\begin{equation}%
\mathbb P (X\!\begin{smallmatrix}(\alpha)\\
\mkern-12mu1\end{smallmatrix} = x_j) = \frac{x^{\alpha}_j}{\sum_{l=1}^n x^{\alpha}_l},~\mathbb P (X\!\begin{smallmatrix}(\alpha)\\
\mkern-12mu2\end{smallmatrix} = x_j \mid X\!\begin{smallmatrix}(\alpha)\\
\mkern-12mu1\end{smallmatrix} = x_i) = \frac{x^{\alpha}_j}{\sum_{l=1}^n x^{\alpha}_l - x^{\alpha}_i},~\ldots
\end{equation}%
Finally, we define
\begin{equation}%
\served{k} = \{c(1), \ldots, c(k)\}
\end{equation}%
as the set of the first $k$ customers served. The following lemma is the first step in understanding the structure of the arrival process:
\begin{lemma}[Size-biased reordering of the arrivals]\label{lem:size_biased_reordering}%
The order of appearance of customers is the $\alpha$-size-biased ordering of their service times. In other words,
\begin{equation}%
\PS(c(j) = i \mid \served{j-1}) = \frac{S_i^{\alpha}}{\sum_{l\nin\served{j-1}}S^{\alpha}_l}.
\end{equation}%
\end{lemma}%
\begin{proof}
Conditioned on $(S_l)_{l=1}^n$, the arrival times are independent exponential random variables. By basic properties of exponentials, we have
\begin{align}\label{eq:size_biased_reordering}%
\PS&(c(j) = i \mid \served{j-1}) = \PS(\min\{T_l:~l\nin \served{j-1}\} = T_i\mid \served{j-1}) = \frac{S_i^{\alpha}}{\sum_{l\nin\served{j-1}}S^{\alpha}_l},
\end{align}%
as desired.
\end{proof}
We remark that \eqref{eq:size_biased_reordering} differs from the classical size-biased reordering in that the weights are a \emph{non-linear} function of the $(S_i)_{i=1}^n$.
The next lemma is crucial, establishing stochastic domination between the service requirements of the customers in order of appearance. \bl{In our definition of the queueing process \eqref{eq:queue_length_process}--\eqref{eq:number_arrivals_during_one_service}, we do not keep track of the service requirements of the customers that join the queue, but only of their arrival times \eqref{eq:arrival_times}. Therefore, at the start of service, a customer's service requirement is a random variable that depends on the arrival time relative to the remaining customers. Lemma \ref{lem:size_biased_reordering} then gives the precise distribution of the service requirement of the $j$-th customer entering service.}

Recall that $X$ stochastically dominates $Y$ (with notation $Y\preceq X$) if and only if there exists a probability space $(\bar{\Omega},\bar{ \mathcal F}, \bar{\mathbb P})$ and two random variables $\bar{X}$, $\bar{Y}$ defined on $\bar{\Omega}$ such that $\bar{X}\sr{\mathrm d}{=} X$, $\bar{Y}\sr{\mathrm d}{=} Y$ and $\bar{\mathbb P}(\bar Y\leq \bar X) = 1$.
\begin{lemma}\label{lem:expectation_service_times_ordering}
Assume that $\alpha>0$. Let $f:\mathbb R^+\rightarrow\mathbb R$ be a function such that $\E[f(S)S^{\alpha}]<\infty$. Then there exists a constant $C_{f, \sss S}$ such that almost surely, for $n$ large enough,
\begin{equation}\label{eq:expectation_service_times_ordering}
\ES[f(S_{c(k)})]\leq C_{f,\sss S}<\infty,
\end{equation}
uniformly in $k\leq c n$, for a fixed $c\in(0,1)$.
\end{lemma}
\begin{proof}
We compute explicitly
\begin{align}%
\ES[f(S_{c(k)})] &= \ES\Big[\frac{\sum_{j\nin\served{k-1}}f(S_j)S_j^{\alpha}}{\sum_{j\nin\served{k-1}}S_j^{\alpha}}\Big]\nnl
& = \ES\Big[\frac{\sum_{j\in[n]}f(S_j)S_j^{\alpha}- \sum_{j\in\served{k}}f(S_j)S_j^{\alpha}}{\sum_{j\nin\served{k-1}}S_j^{\alpha}}\Big]\nnl
& \leq \ES\Big[\frac{1}{\sum_{j\nin\served{k-1}}S_j^{\alpha}}\Big]\sum_{j\in[n]}f(S_j)S_j^{\alpha}.
\end{align}%
We have the almost sure bound
\begin{align}%
\frac{1}{\sum_{j\nin\served{k-1}}S_j^{\alpha}} &= \frac{1}{\sum_{j\in[n]}S_j^{\alpha} - \sum_{j\in\served{k-1}}S_j^{\alpha}}  \leq \frac{1}{\sum_{j\in[n]}S_j^{\alpha} - \sum_{j\in\served{k-1}}S_j^{\alpha}} \nnl
& \leq \frac{1}{\sum_{j\in[n]}S_j^{\alpha} - \sum_{j=1}^{k-1}S_{(n-j+1)}^{\alpha}} = \frac{1}{\sum_{j=1}^{n-k+1}S_{(j)}^{\alpha}},
\end{align}%
where $S_{(1)}^{\alpha}\leq S_{(2)}^{\alpha}\leq\ldots\leq S_{(n)}^{\alpha}$ denote the order statistics of the finite sequence $(S_i^{\alpha})_{i\in[n]}$.
There exists $p\in(0,1)$ such that $n-k+1\geq p n$, for large enough $n$. Consequently,
\begin{equation}%
\frac{1}{\sum_{j\nin\served{k-1}}S_j^{\alpha}} \leq \frac{1}{\sum_{j=1}^{\lfloor pn\rfloor}S_{(j)}^{\alpha}},
\end{equation}%
so that we have
\begin{equation}%
\ES[f(S_{c(k)})] \leq \frac{\sum_{j\in[n]}f(S_j)S_j^{\alpha}}{\sum_{j=1}^{\lfloor pn\rfloor}S_{(j)}^{\alpha}}.
\end{equation}%
%
%
Let us denote by $\xi_p$ the $p$-th quantile of the distribution $F_{\sss S}(\cdot)$ and let us assume, without loss of generality, that $f_{\sss S}(\xi_p) > 0$.

Note that $S_{(\lfloor np\rfloor)} = F_{n, \sss S}^{-1}(\lfloor np \rfloor/n)$, where $F_{n,\sss S}(t) = \sum_{i=1}^n\mathds 1_{\{S_i\leq t\}}/n$ is the empirical distribution function of the $(S_i)_{i=1}^n$, and $\xi_p = F_{\sss S}^{-1}(p)$. Indeed, the assumption $f_{\sss S}(\xi_p)>0$ implies that $F_{\sss S}(\cdot)$ is invertible in a neighborhood of $\xi_p$.
We have that, as $n\rightarrow\infty$,
%
\begin{equation}%
S_{(\lfloor np\rfloor)}\asconv \xi_p.
\end{equation}%
In particular, as $n\rightarrow\infty$,
\begin{equation}%
\frac{1}{n}\Big\vert \sum_{j\in[n]} S_j \mathds 1_{\{S_j\leq \xi_p\}} - \sum_{j\in[n]}S_j\mathds 1_{\{S_j\leq S_{(\lfloor pn\rfloor)}\}} \Big\vert\asconv 0.
\end{equation}%
Therefore, by the strong Law of Large Numbers, as $n\rightarrow\infty$,
\begin{equation}%
\frac{\sum_{j=1}^{\lfloor pn\rfloor}S_{(j)}}{n}\asconv \E[S\mathds 1_{\{S\leq \xi_p\}}].
\end{equation}%
Then, choosing $C_{n,f,\sss S} = \E[f(S)S^{\alpha}]/\E[S\mathds 1_{\{S\leq \xi_p\}}] + \varepsilon$, for an arbitrary $\varepsilon>0$, gives the desired result.
\end{proof}
If $\alpha >0$, as is the case in our setting, the proof of Lemma \ref{lem:expectation_service_times_ordering} shows that, uniformly in $k=O(n^{2/3})$,
\begin{align}
\ES[f(S_{c(k)})]&\leq \frac{\sum_{j\in[n]}f(S_j)S_j^{\alpha}}{\sum_{j=1}^{\lfloor pn\rfloor}S_{(j)}^{\alpha}}\nnl
&= \frac{\sum_{j\in[n]}f(S_j)S_j^{\alpha}}{\sum_{j=1}^{n}S_{(j)}^{\alpha}}\Big(1 + \frac{\sum_{j=\lfloor pn\rfloor}^{n}S_{(j)}^{\alpha}}{\sum_{j=1}^{\lfloor pn\rfloor}S_{(j)}^{\alpha}}\Big),
\end{align}
and therefore
\begin{equation}\label{eq:average_first_customer_service_time_is_larger}%
\ES[f(S_{c(k)})]\leq \ES[f(S_{c(1)})](1+\OPS(1)).
\end{equation}%
If $f(\cdot)$ is an increasing function, \eqref{eq:average_first_customer_service_time_is_larger} makes precise the intuition that, if $\alpha>0$, customers with larger job sizes join the queue earlier. We will often make use of the expression \eqref{eq:average_first_customer_service_time_is_larger}.

The following lemma will often prove useful in dealing with  sums over a random index set:
\begin{lemma}[Uniform convergence of random sums]\label{lem:unif_convergence_random_sums}%
Let $(S_j)_{j=1}^{n}$ be a sequence of positive random variables such that $\mathbb E[S^{2+\alpha}] < + \infty$, for $\alpha\in(0,1)$. Then,
\begin{equation}%
\sup_{\substack{\mathcal X\subseteq[n]\\ \vert \mathcal X \vert = \OP (n^{2/3})}}\frac{1}{n}\sum_{j\in \mathcal X}S_j^{\alpha} = \oP(1).
\end{equation}%
\end{lemma}%
\begin{proof}%
By Lemma \ref{lem:max_iid_random_variables}, $\max_{j\in[n]} S_j^{\alpha} = \oP(n^{\alpha/(2+\alpha)})$. This gives
\begin{align}%
\sup_{\substack{\mathcal X\subseteq[n]\\ \vert \mathcal X \vert = \OP (n^{2/3})}}\frac{1}{n}\sum_{j\in \mathcal X}S_j^{\alpha} \leq \frac{\max_{j\in[n]}S_j^{\alpha}}{n^{1/3}}\OP(1) = \oP (n^{\frac{\alpha-2/3-\alpha/3}{2+\alpha}}) = \oP(n^{\frac{2}{3}\frac{\alpha-1}{2+\alpha}}).
\end{align}%
Since $\alpha-1 \leq 0$ by assumption, the claim is proven.
\end{proof}%
We now focus on the $i$-th customer joining the queue (for $i$ large) and characterize the distribution of its service time. In particular, for $\alpha>0$ this is different from $S_i$.
\begin{lemma}[Size-biased distribution of the service times]\label{lem:size_biased_distribution}%
For every bounded, real-valued continuous function $f(\cdot)$, as $n\rightarrow\infty$,
\begin{equation}\label{eq:size_biased_distribution}%
\ES[f(S_{c(i)})\mid \mathcal F _{i-1}]\Pconv\frac{\mathbb E[f(S)S^{\alpha}]}{\mathbb E[S^{\alpha}]},
\end{equation}%
uniformly for $i=\OPS(n^{2/3})$. Moreover, as $n\rightarrow\infty$,
\begin{equation}%
\ES[f(S_{c(i)})]\rightarrow\frac{\mathbb E[f(S)S^{\alpha}]}{\mathbb E[S^{\alpha}]},\qquad \mathrm{for}~i=\OPS(n^{2/3}).
\end{equation}%
\end{lemma}%
\begin{proof}
First note that
\begin{align}
\ES[f(S_{c(i)}) \mid \mathcal F_{i-1}] &= \sum_{j\nin \served{i-1}}f(S_j)\PS(c(i) = j \mid\mathcal F_{i-1}) = \sum_{j\nin \served{i-1}}\frac{f(S_j)S_j^{\alpha}}{\sum_{l\nin \served{i-1}}S_l^{\alpha}}.
%
\end{align}
This can be further decomposed as
\begin{equation}%
\ES[f(S_{c(i)}) \mid \mathcal F_{i-1}] = \frac{\sum_{j=1}^nf(S_j)S_j^{\alpha} - \sum_{j\in\served{i-1}}f(S_j)S_j^{\alpha}}{\sum_{l=1}^nS_l^{\alpha} - \sum_{l\in\served{i-1}}S_l^{\alpha}}.
\end{equation}%
Since $\vert\served{i-1}\vert = i-1$ and $i = \OP(n^{2/3})$,  by the Law of Large Numbers and Lemma \ref{lem:unif_convergence_random_sums},
\begin{align}%
\frac{\sum_{j\nin\served{i-1}}f(S_j)S_j^{\alpha}}{n}\Pconv \E[f(S)S^{\alpha}],\quad\frac{\sum_{l\nin\served{i-1}}S_l^{\alpha}}{n}\Pconv \E[S^{\alpha}].
\end{align}%
uniformly in $i = \OP(n^{2/3})$. This gives the first claim.

Furthermore, we bound $\ES[f(S_{c(i)})\mid \mathcal F_{i-1}]$ as
\begin{equation}%
\ES[f(S_{c(i)})\mid \mathcal F_{i-1}] = \sum_{j\nin \served{i-1}}\frac{f(S_j)S_j^{\alpha}}{\sum_{l\nin \served{i-1}}S_l^{\alpha}} \leq \sup_{x\geq0}f(x)<\infty.
\end{equation}%
Since $\ES[f(S_{c(i)})] = \ES[\ES[f(S_{c(i)}) \mid \mathcal F_{i-1}]]$, using \eqref{eq:size_biased_distribution} and the Dominated Convergence \mbox{Theorem} the second claim follows.

\end{proof}
In Lemma \ref{lem:size_biased_distribution} we have studied the distribution of the service time of the $i$-th customer, and we now focus on its (conditional) moments. The following lemma should be interpreted as follows: Because of the size-biased re-ordering of the customer arrivals, the service time of the $i$-th customer being served (for $i$ large) is highly concentrated.
\begin{lemma}\label{lem:1+alphaConditionalMoment}%
For any fixed $\gamma\in[-1,1]$,
\begin{equation}\label{eq:1+alphaConditionalMoment}%
\ES[S_{c(i)}^{1+\gamma} \mid \mathcal F_{i-1}] = \frac{\mathbb E[S^{1+\gamma+\alpha}]}{\mathbb E[S^{\alpha}]} + o_{\mathbb P}(1)\quad\mathrm{for}~i = \OPS(n^{2/3}),
\end{equation}%
where the error term is uniform in $i= O_{\PS}(n^{2/3})$.
%
Moreover, the convergence holds in $L^1$, i.e.
\begin{equation}\label{eq:1+alphaConditionalMoment_L1_convergence}%
\ES\Big[\Big\vert\ES[ S_{c(i)}^{1+\gamma}\mid \mathcal F_{i-1}] - \frac{\mathbb E[S^{1+\gamma+\alpha}]}{\mathbb E[S^{\alpha}]}\Big\vert\Big] = \oP(1),
\end{equation}%
uniformly in $i=\OPS(n^{2/3})$.
\end{lemma}%
%
\begin{proof}
In order to apply Lemma \ref{lem:size_biased_distribution}, we first split
\begin{equation}%
S_{c(i)}^{1+\gamma} = (S_{c(i)}\wedge K)^{1+\gamma} + ((S_{c(i)} - K)^+)^{1+\gamma},
\end{equation}%
where $K>0$ is arbitrary, so that
\begin{equation}%
\ES[S_{c(i)}^{1+\gamma} \mid \mathcal F_{i-1}] = \ES[(S_{c(i)}\wedge K)^{1+\gamma} \mid \mathcal F_{i-1}] + \ES[((S_{c(i)} - K)^+)^{1+\gamma} \mid \mathcal F_{i-1}].
\end{equation}%
The first term is bounded, and therefore converges to $\E[(S\wedge K)^{1+\gamma}S^{\alpha}]/\E[S^{\alpha}]$ by Lemma \ref{lem:size_biased_distribution}.
The second term is bounded through Markov's inequality, as
\begin{equation}%
\PS (\ES[((S_{c(i)} - K)^+)^{1+\gamma} \mid \mathcal F_{i-1}] \geq \varepsilon) \leq \frac{\ES[((S_{c(i)}-K)^+)^{1+\gamma}]}{\varepsilon}.
\end{equation}%
Next we apply Lemma \ref{lem:expectation_service_times_ordering} with $f(x) = f_{\sss K}(x) =((x-K)^+)^{1+\gamma}$,
\begin{equation}\label{eq:1+alpha_error_term_bound}%
\ES[((S_{c(i)}-K)^+)^{1+\gamma}] \leq  C_{f_{\sss K},\sss S}.
\end{equation}%
Therefore,
\begin{align}\label{eq:1+alpha_error_term_bound_second}%
\Big\vert \ES[S_{c(i)}^{1+\gamma} \mid \mathcal F_{i-1}] - \frac{\E[S^{1+\gamma+\alpha}]}{\E[S^{\alpha}]}\Big\vert &\leq \Big\vert\ES[(S_{c(i)}\wedge K)^{1+\gamma} \mid \mathcal F_{i-1}] - \frac{\E[S^{1+\gamma+\alpha}]}{\E[S^{\alpha}]}\Big\vert + C_{f_{\sss K},\sss S}.
\end{align}%
The proof of Lemma \ref{lem:expectation_service_times_ordering} shows that, for any $\varepsilon>0$, $\lim_{K\rightarrow\infty}C_{f_{\sss K},\sss S} \leq \varepsilon$, thus $\lim_{K\rightarrow\infty}C_{f_{\sss K},\sss S} = 0$. Therefore, by letting $K\rightarrow\infty$ in \eqref{eq:1+alpha_error_term_bound_second},  \eqref{eq:1+alphaConditionalMoment} follows. Next, we split
\begin{align}%
\ES\Big[\Big\vert \ES[S_{c(i)}^{1+\gamma} \mid \mathcal F_{i-1}] - \frac{\E[S^{1+\gamma+\alpha}]}{\E[S^{\alpha}]}\Big\vert\Big] &\leq \ES\Big[\Big\vert(S_{c(i)}\wedge K)^{1+\gamma} - \frac{\E[S^{1+\gamma+\alpha}]}{\E[S^{\alpha}]}\Big\vert\Big]\nnl
&\quad+\ES[((S_{c(i)} - K)^+)^{1+\gamma}].
\end{align}%
The second term can be bounded as in \eqref{eq:1+alpha_error_term_bound}. For the first term,
\begin{align}%
\ES&\Big[\Big\vert(S_{c(i)}\wedge K)^{1+\gamma} - \frac{\E[S^{1+\gamma+\alpha}]}{\E[S^{\alpha}]}\Big\vert\Big] \leq \Big\vert\frac{\sum_{j=1}^n(S_{j}\wedge K)^{1+\gamma}S_j^{\alpha}}{\sum_{j=1}^n S_j^{\alpha}} - \frac{\E[S^{1+\gamma+\alpha}]}{\E[S^{\alpha}]}\Big\vert \nnl
&+\ES\Big[\Big\vert \frac{\sum_{j=1}^n(S_j\wedge K)^{1+\gamma}S_j^{\alpha}\sum_{l\in\served{i-1}}S_l^{\alpha}  }{(\sum_{j=1}^nS_j^{\alpha})^2}\Big\vert\Big] + \ES\Big[\Big\vert\frac{\sum_{l=1}^n S_l^{\alpha}\sum_{j\in \served{i-1}}(S_j\wedge K)^{1+\gamma}S_j^{\alpha}}{(\sum_{j=1}^nS_j^{\alpha})^2} \Big\vert\Big],
\end{align}%
where we have used that $\vert(a-b)/(c-d) -a/c\vert\leq ad/c^2 + bc/c^2$, for positive $a$, $b$, $c$, $d$. The second and third terms converge uniformly over $i=O_{\PS}(n^{2/3})$ by Lemma \ref{lem:unif_convergence_random_sums}. Summarizing,
\begin{align}%
\ES\Big[\Big\vert \ES[S_{c(i)}^{1+\gamma} \mid \mathcal F_{i-1}] - \frac{\E[S^{1+\gamma+\alpha}]}{\E[S^{\alpha}]}\Big\vert\Big] &\leq \Big\vert\frac{\sum_{j=1}^n(S_{j}\wedge K)^{1+\gamma}S_j^{\alpha}}{\sum_{j=1}^n S_j^{\alpha}} - \frac{\E[S^{1+\gamma+\alpha}]}{\E[S^{\alpha}]}\Big\vert \nnl
&\quad+ \frac{\sum_{l=1}^n((S_l-K)^+)^{1+\gamma}}{\sum_{j=1}^n S_j^{\alpha}} + \oP(1).
\end{align}%
Letting first $n\rightarrow\infty$ and then $K\rightarrow\infty$, \eqref{eq:1+alphaConditionalMoment_L1_convergence} follows.
\end{proof}

We will make use of Lemma \ref{lem:1+alphaConditionalMoment} several times throughout the proof, with the specific choices $\gamma \in \{0,\alpha,1\}$. The following lemma is of central importance in the proof of the uniform convergence of the quadratic part of the drift:
\begin{lemma}\label{lem:size_biased_service_times_unif_convergence}%
As $n\rightarrow\infty$,
\begin{equation}\label{eq:size_biased_service_times_unif_convergence}%
n^{-2/3}\sup_{j\leq tn^{2/3}}\Big\vert \sum_{i=1}^{j}\Big(S_{c(i)}^{1+\alpha} - \frac{\mathbb E[S^{1+2\alpha}]}{\mathbb E[S]}\Big)\Big\vert \stackrel{\mathbb P}{\rightarrow} 0.
\end{equation}%
\end{lemma}%
\begin{proof}
By Lemma \ref{lem:1+alphaConditionalMoment}, \eqref{eq:size_biased_service_times_unif_convergence} is equivalent to
\begin{equation}%
n^{-2/3}\sup_{j\leq tn^{2/3}}\Big\vert \sum_{i=1}^{j}\Big(S_{c(i)}^{1+\alpha}- \E[S_{c(i)}^{1+\alpha}\mid \mathcal F_{i-1}]\Big)\Big\vert \stackrel{\mathbb P}{\rightarrow} 0.
\end{equation}%
We split the event space  and separately bound
\begin{equation}\label{eq:size_biased_service_times_unif_convergence_smaller_Kn}%
n^{-2/3}\sup_{j\leq tn^{2/3}}\Big\vert \sum_{i=1}^{j}\Big(S_{c(i)}^{1+\alpha}\mathds 1_{\{S^{1+\alpha}_{c(i)}\leq K_n\}} - \E[S_{c(i)}^{1+\alpha}\mathds 1_{\{S^{1+\alpha}_{c(i)}\leq K_n\}} \mid \mathcal F_{i-1}]\Big)\Big\vert
\end{equation}%
and
\begin{equation}\label{eq:size_biased_service_times_unif_convergence_greater_Kn}%
n^{-2/3}\sup_{j\leq tn^{2/3}}\Big\vert \sum_{i=1}^{j}\Big(S_{c(i)}^{1+\alpha}\mathds 1_{\{S^{1+\alpha}_{c(i)}> K_n\}} - \E[S_{c(i)}^{1+\alpha}\mathds 1_{\{S^{1+\alpha}_{c(i)}> K_n\}} \mid \mathcal F_{i-1}]\Big)\Big\vert,
\end{equation}%
for a sequence $(K_n)_{n\geq1}$ that we choose later on and is such that $K_n\rightarrow\infty$. We start with \eqref{eq:size_biased_service_times_unif_convergence_smaller_Kn}. Since the sum inside the absolute value is a martingale as a function of $j$, \eqref{eq:size_biased_service_times_unif_convergence_smaller_Kn} can be bounded through Doob's $L^p$ inequality \cite[Theorem 11.2]{klenke2008probability} with $p=2$ as
\begin{align}%
\PS\Big(\sup_{j\leq tn^{2/3}}&\Big\vert \sum_{i=1}^{j}\Big(S_{c(i)}^{1+\alpha}\mathds 1_{\{S^{1+\alpha}_{c(i)}\leq K_n\}} - \ES[S_{c(i)}^{1+\alpha}\mathds 1_{\{S^{1+\alpha}_{c(i)}\leq K_n\}} \mid \mathcal F_{i-1}]\Big)\Big\vert \geq\varepsilon n^{2/3}\Big) \nnl
&\leq \frac{1}{\varepsilon n^{4/3}}\ES\Big[\sum_{i=1}^{tn^{2/3}}(S_{c(i)}^{1+\alpha}\mathds 1_{\{S^{1+\alpha}_{c(i)}\leq K_n\}} -\ES[S_{c(i)}^{1+\alpha}\mathds 1_{\{S^{1+\alpha}_{c(i)}\leq K_n\}} \mid \mathcal F_{i-1}])^2\Big]\nnl
&\leq \frac{2}{\varepsilon n^{4/3}}\sum_{i=1}^{tn^{2/3}}\ES[S_{c(i)}^{2+2\alpha} \mathds 1_{\{S^{1+\alpha}_{c(i)}\leq K_n\}}] \leq \frac{2}{\varepsilon n^{4/3}}\sum_{i=1}^{tn^{2/3}}K_n^{2\alpha}\ES[S_{c(i)}^{2}] .
\end{align}%
Lemma \ref{lem:1+alphaConditionalMoment} allows us to  approximate $\ES[S_{c(i)}^{2}]$ uniformly by $\frac{\E[S^{2+\alpha}]}{\E[S^{\alpha}]}$. Thus, we get
\begin{align}%
\frac{2}{\varepsilon n^{4/3}}\sum_{i=1}^{tn^{2/3}}\Big(K_n^{2\alpha}\frac{\E[S^{2+\alpha}]}{\E[ S^{\alpha}]} +\oP(1)\Big) &= \frac{t K_n^{2\alpha}}{\varepsilon n^{2/3}} \OP(1),
\end{align}%
which converges to zero as $n\rightarrow\infty$ if and only if  $K_n^{\alpha}/n^{1/3}$ does.
We now turn to \eqref{eq:size_biased_service_times_unif_convergence_greater_Kn} and apply Doob's $L^1$ martingale inequality \cite[Theorem 11.2]{klenke2008probability} to obtain
\begin{align}\label{eq:size_biased_service_times_unif_convergence_greater_Kn_second}%
&\PS\Big(\sup_{j\leq tn^{2/3}}\Big\vert \sum_{i=1}^{j}\Big(S_{c(i)}^{1+\alpha}\mathds 1_{\{S^{1+\alpha}_{c(i)}> K_n\}} - \ES[S_{c(i)}^{1+\alpha}\mathds 1_{\{S^{1+\alpha}_{c(i)}> K_n\}} \mid \mathcal F_{i-1}]\Big)\Big\vert\geq\varepsilon n^{2/3}\Big) \nnl
&\leq \frac{1}{\varepsilon n^{2/3}}\ES\Big[\Big\vert\sum_{i=1}^{tn^{2/3}}(S_{c(i)}^{1+\alpha}\mathds 1_{\{S^{1+\alpha}_{c(i)}> K_n\}} -\ES[S_{c(i)}^{1+\alpha}\mathds 1_{\{S^{1+\alpha}_{c(i)}> K_n\}} \mid \mathcal F_{i-1}])\Big\vert\Big]\nnl
&\leq \frac{2}{\varepsilon n^{2/3}}\sum_{i=1}^{tn^{2/3}}\ES[S_{c(i)}^{1+\alpha}\mathds 1_{\{S^{1+\alpha}_{c(i)}> K_n\}} ] \leq \frac{2}{\varepsilon n^{2/3}}\sum_{i=1}^{tn^{2/3}}\ES[S_{c(1)}^{1+\alpha}\mathds 1_{\{S^{1+\alpha}_{c(1)}> K_n\}} ](1+\OPS(1))\nnl
& = \frac{2t}{\varepsilon}\ES[S_{c(1)}^{1+\alpha}\mathds 1_{\{S^{1+\alpha}_{c(1)}> K_n\}}](1+\OPS(1))=\oP(1).
\end{align}%
We have used Lemma \ref{lem:1+alphaConditionalMoment} in the second inequality, and Lemma \ref{lem:expectation_service_times_ordering} with $f(x) = x^{1+\alpha}\mathds 1_{\{x^{1+\alpha} > K_n\}}$ in the third. The right-most term in \eqref{eq:size_biased_service_times_unif_convergence_greater_Kn_second} is $\oP(1)$ as $n\rightarrow\infty$ by the strong Law of Large Numbers. Note that this side of the bound does not impose additional conditions on $K_n$, so that, if we take $K_n= n^{c}$, it is sufficient that $c<\frac{1}{3\alpha}$, with the convention that $\frac{1}{0}=\infty$.
\end{proof}

%
%

We conclude this section with a technical lemma concerning error terms  in the computations of quadratic variations. Denote the density (resp.~distribution function) of a rate $\lambda$ exponential random variable by $\fE(\cdot)$ (resp.~$\FE(\cdot)$):
\begin{lemma}\label{lem:n_squared_err_term}%
We have
\begin{equation}\label{eq:n_squared_err_term_statement}%
\ES\Big[\sum_{h,q\in[n]}\Big\vert\FE \Big(\frac{S_{c(i)}S_h^{\alpha}}{n}\Big) - \frac{\lambda S_{c(i)}S_h^{\alpha}}{n}\Big\vert \Big\vert \FE \Big( \frac{S_{c(i)}S_q^{\alpha}}{n}\Big) - \frac{\lambda S_{c(i)}S_q^{\alpha}}{n}\Big\vert \mid \mathcal F_{i-1}\Big] = o_{\mathbb P}(1)
\end{equation}%
uniformly in $i=O(n^{2/3})$.
\end{lemma}%
\begin{proof}%
%
Since $\vert\FE (x) - x\vert = O(x^2)$, the bound $\vert\lambda S_{c(i)}S_h^{\alpha}/n-\FE(S_{c(i)}S_h^{\alpha}/n)\vert \leq C(S_{c(i)}S_h^{\alpha}/n)^{1+\varepsilon}$ holds almost surely for $0<\varepsilon<1$ and $C>0$, which gives
\begin{align}%
\lambda^2\sum_{h,q\in[n]}\ES\Big[\Big( \frac{S_{c(i)}S_h^{\alpha}}{n}\Big)^{1+\varepsilon} \Big(  \frac{S_q^{\alpha}S_{c(i)}}{n}\Big)^{1+\varepsilon} \mid \mathcal F_{i-1}\Big]= \frac{\lambda^2}{n^{2+2\varepsilon}}\sum_{h,q\in[n]} \ES[ S_{c(i)}^{2+2\varepsilon} \mid \mathcal F_{i-1}]  S_h^{\alpha(1+\varepsilon)}   S_q^{\alpha(1+\varepsilon)}
\end{align}%
Therefore,
\begin{align}\label{eq:n_squared_err_term_final}%
\lambda^2&\sum_{h,q\in[n]}\ES\Big[\Big( \frac{S_{c(i)}S_h^{\alpha}}{n}\Big)^{1+\varepsilon} \Big(  \frac{S_q^{\alpha}S_{c(i)}}{n}\Big)^{1+\varepsilon}\mid \mathcal F_{i-1}\Big] \nnl
&\leq \frac{\lambda^2}{n^{2+2\varepsilon}}\max_{j\in[n]}S_j^{2\varepsilon}\ES[ S_{c(i)}^{2} \mid \mathcal F_{i-1}]\sum_{h,q\in[n]} S_h^{\alpha(1+\varepsilon)}  S_q^{\alpha(1+\varepsilon)} \nnl
&\leq \frac{\lambda^2\E[S^{2+\alpha}]}{\E[S^{\alpha}]}\frac{\max_{j\in[n]}S_j^{2\varepsilon}}{n^{2\varepsilon}}\frac{1}{n^2}\sum_{h,q\in[n]} S_h^{\alpha(1+\varepsilon)}  S_q^{\alpha(1+\varepsilon)} +\oP(1),
\end{align}%
where in the last step we used Lemma \ref{lem:1+alphaConditionalMoment}. Note that, since $\mathbb E[S^{2+\alpha}]<\infty$, by Lemma \ref{lem:max_iid_random_variables} $\max_{j\in[n]}S_j^{2\varepsilon} = o_{\mathbb P}(n^{2\varepsilon/(2+\alpha)})$.
The right-most term in \eqref{eq:n_squared_err_term_final} then tends to zero as $n$ tends to infinity as long as $0<\varepsilon<\min\{1,2/\alpha\}$.
\end{proof}%

\section{Proving the scaling limit}\label{sec:proof_main_result}
We first establish some preliminary estimates on $N_n(\cdot)$ that will \bl{be} crucial for the proof of convergence. We will upper bound the process $N_n(\cdot)$ by a simpler process $N_n^{\sss U}(\cdot)$ in such a way that the increments of $N_n^{\sss U}(\cdot)$ almost surely dominate the increments of $N_n(\cdot)$. We also show that, after rescaling, $N_n^{\sss U}(\cdot)$ converges in distribution to $W(\cdot)$. The process $N_n^{\sss U}(\cdot)$ is defined as $N_n^{\sss U}(0) = N_n(0)$, and
\begin{equation}\label{eq:pre_reflection_queue_length_process_upper_bound}%
N_n^{\sss U}(k) = N_n^{\sss U}(k-1) + A^{\sss U}_n(k) -1,
\end{equation}%
where
\begin{equation}\label{eq:arrivals_upper_bound}%
A^{\sss U}_n(k) = \sum_{i\nin\served{k}} \mathds 1 _{\{T_i \leq c_{n,\beta}S_{c(k)}/n\}},
\end{equation}%
with
\begin{equation}
c_{n,\beta}=1+\beta n^{-1/3},
\end{equation}%
and
\begin{equation}%
T_{i} \sr{\mathrm d}{=} \textrm{exp}_i(\lambda S_i^{\alpha}).
\end{equation}%
An interpretation of the process $N_n^{\sss U}(\cdot)$ is that customers are not removed from the pool of potential customers until they have been served. Therefore, a customer could potentially join the queue more than once. We couple the processes $N_n(\cdot)$ and $N_n^{\sss U}(\cdot)$ as follows. Consider a sequence of arrival times $(T_i)_{i=1}^{\infty}$ and of service times $(S_i)_{i=1}^{\infty}$, then define $A_n(\cdot)$ as  \eqref{eq:number_arrivals_during_one_service} and $A_n^{\sss U}(\cdot)$ as  \eqref{eq:arrivals_upper_bound}. With this coupling we have that, almost surely,
\begin{equation}%
A_n(k) \leq A_n^{\sss U}(k) \qquad \forall~k \geq 1.
\end{equation}%
Consequently,
\begin{equation}\label{eq:pre_reflection_process_domination}%
N_n (k) \leq N_n^{\sss U}(k)\qquad\forall k\geq 0,
\end{equation}%
and
\begin{equation}\label{eq:reflected_process_domination}%
Q_n(k) = \phi(N_n) (k) \leq\phi(N_n^{\sss U})(k) =: Q_n^{\sss U}(k) \qquad\forall k\geq 0,
\end{equation}%
almost surely.

While in general only the upper bounds \eqref{eq:pre_reflection_process_domination} and \eqref{eq:reflected_process_domination} hold, the processes $N_n(\cdot)$ and $N_n^{\sss U}(\cdot)$ (resp.~$Q_n(\cdot)$ and $Q_n^{\sss U}(\cdot)$) turn out to be, very close to each other. We start by proving results for $N_n^{\sss U}(\cdot)$ and $Q_n^{\sss U}(\cdot)$ because they are easier to treat, and only then we are able to prove that identical results hold for $N_n(\cdot)$ and $Q_n(\cdot)$.

In fact, we introduce the upper bound $N_n^{U}(\cdot)$ to deal with the complicated index set for the summation in \eqref{eq:number_arrivals_during_one_service}. The difficulty arises as follows: in order to estimate $N_n(\cdot)$ one has to estimate $A_n(\cdot)$. To do this, one has to separately (uniformly) bound each element in the sum, and also estimate the number of elements in the sum. The first goal is accomplished, for example, through Lemma \ref{lem:1+alphaConditionalMoment}, while for the second the crude upper bound $n$ is not strict enough. However, estimating $\vert \nu_k\vert$ requires an estimate on $N_n(\cdot)$ itself, as \eqref{eq:cardinality_nu} shows. To solve this circularity, we introduce a bootstrap argument: first, we upper bound $N_n(\cdot)$ and we obtain estimates on the upper bound, from this follows an estimate on $\vert \nu_k\vert$, and this in turn allows us to estimate $N_n(\cdot)$.

This technique can be applied to solve a recently found technical issue in the proof of the main result of \cite{bhamidi2010scaling}. The authors in \cite{bhamidi2010scaling} prove convergence of a process which upper bounds the exploration process of the graph. Therefore, their main result is analogous to Theorem \ref{MainTheorem_U}. However, a further step is required to complete the proof of convergence of the exploration process, and this is provided by our approach.

\begin{theorem}[Convergence of the upper bound]\label{MainTheorem_U}
\begin{equation}
n^{-1/3}N_n^{\sss U}( t n^{2/3})\stackrel{\mathrm{d}}{\rightarrow} W(t)\qquad\mathrm{in}~(\mathcal D, J_1)~\mathrm{as}~n\rightarrow\infty,
\end{equation}
where $W(\cdot)$ is the diffusion process in \eqref{eq:main_theorem_delta_G_1_diffusion_definition}.
In particular,
\begin{equation}%
n^{-1/3}\phi(N_n^{\sss U})( t n^{2/3})\stackrel{\mathrm{d}}{\rightarrow} \phi(W)(t)\qquad\mathrm{in}~(\mathcal D, J_1)~\mathrm{as}~n\rightarrow\infty.
\end{equation}%
\end{theorem}
The next section is dedicated to the proof of Theorem \ref{MainTheorem_U}.
}
\subsection{Convergence of the upper bound}\label{sec:proof_thm_U}
We use a classical martingale decomposition followed by a martingale FCLT.
The process $N_n^{\sss U}(\cdot)$ in \eqref{eq:pre_reflection_queue_length_process_upper_bound} can be decomposed as $N_n^{\sss U}(k) = M_n^{\sss U}(k) + C_n^{\sss U}(k)$, where $M_n^{\sss U}(\cdot)$ is a martingale and $C_n^{\sss U}(\cdot)$ is a drift term, as follows:
\begin{align}%
M^{\sss U}_n(k) &= \sum_{i=1}^k(A^{\sss U}_n(i) - \ES[A^{\sss U}_n(i) \mid \mathcal F_{i-1}]), \nnl
C^{\sss U}_n(k) &=  \sum_{i=1}^k (\ES[A^{\sss U}_n(i) \mid \mathcal F_{i-1}]-1).
\end{align}%
 Moreover, $(M^{\sss U}_n(k))^2$ can  be written as $(M^{\sss U}_n(k))^2 = Z^{\sss U}_n(k) + B^{\sss U}_n(k)$ with $Z^{\sss U}_n(k)$  a martingale and $B^{\sss U}_n(k)$ the compensator, or quadratic variation, of $M^{\sss U}_n(k)$ given by
\begin{equation}
B^{\sss U}_n(k) = \sum_{i=1}^k(\ES[(A^{\sss U}_n(i))^2 \mid \mathcal F_{i-1}] - \ES[A^{\sss U}_n(i) \mid \mathcal F_{i-1}]^2).
\end{equation}%

In order to prove convergence of $N^{\sss U}_n(\cdot)$ we separately prove convergence of $C^{\sss U}_n(\cdot)$ and of $M^{\sss U}_n(\cdot)$. We prove the former directly, and the latter by applying the martingale FCLT \cite[Theorem~7.1.4]{MarkovProcesses}. For this, we need to verify the following conditions:
\begin{enumerate}[label=(\roman*)]
\item \label{size_biased:eq:first_condition_martingale_FCLT} $\sup_{t\leq \bar t} \vert n^{-1/3}C_n^{\sss U}(tn^{2/3})- \beta t + \lambda\frac{\E[S^{1+2\alpha}]}{2\E[S^{\alpha}]} t^2\vert\stackrel{\mathbb P}{\longrightarrow}0,\qquad \forall \bar t\in \mathbb R^+$;
\item \label{size_biased:eq:second_condition_martingale_FCLT}$ n^{-2/3}B_n^{\sss U}(tn^{2/3})\stackrel{\mathbb  P}{\longrightarrow}  \sigma^2  t,\qquad \forall t\in \mathbb R^+$;
\item \label{size_biased:eq:third_condition_martingale_FCLT}$\lim_{n\rightarrow\infty}n^{-2/3} \ES[\sup_{t\leq\bar t} \vert B_n^{\sss U}(tn^{2/3})-B_n^{\sss U}(tn^{2/3}-)\vert]=0,\qquad \forall \bar t\in \mathbb R^+$;
\item \label{size_biased:eq:fourth_condition_martingale_FCLT}$\lim_{n\rightarrow\infty} n^{-2/3} \ES[\sup_{t\leq\bar t} \vert M_n^{\sss U}(tn^{2/3})-M_n^{\sss U}(tn^{2/3}-)\vert^2]=0,\qquad \forall \bar t\in \mathbb R^+$.
\end{enumerate}
\subsubsection{Proof of \ref{size_biased:eq:first_condition_martingale_FCLT} for the upper bound} \label{sec:drift_of_upper_bounding_process_converges_to_zero}
First we obtain an explicit expression for $\mathbb E[A_n^{\sss U}(i) \mid \mathcal F_{i-1}]$, as
\begin{align}\label{eq:conditioned_arrivals_first_eq}%
\ES[A_n^{\sss U}(i) \mid \mathcal F_{i-1}] &= \sum_{j\nin\served{i-1}}\PS(c(i)= j\mid\mathcal F_{i-1})\sum_{l\nin\served{i-1}\cup \{j\}}F_{\sss E}\Big(\frac{c_{n,\beta}S_jS_l^{\alpha}}{n}\Big)\\
&= \sum_{j\nin\served{i-1}}\PS (c(i) = j \mid \mathcal F_{i-1})\sum_{l=1}^n \frac{c_{n,\beta}\lambda S_jS_l^{\alpha}}{n} \nnl
&\quad - \sum_{j\nin\served{i-1}}\PS(c(i) = j \mid \mathcal F_{i-1})\sum_{l\in\served{i-1}\cup \{j\}} \frac{c_{n,\beta}\lambda S_jS_l^{\alpha}}{n}\nnl
&\quad + \sum_{j\nin\served{i-1}}\PS(c(i)= j \mid \mathcal F_{i-1})\sum_{l\nin\served{i-1}\cup \{j\}} \Big( F_{\sss E}\Big(\frac{c_{n,\beta}S_jS_l^{\alpha}}{n}\Big)  - \frac{c_{n,\beta}\lambda S_jS_l^{\alpha}}{n}\Big)\notag
\end{align}%
The third term is an error term. Indeed, for some $\zeta_n \in  [0,S_{c(i)}S_l/n]$,
\begin{align}\label{eq:taylor_error_expanding_FT}%
\ES&\Big[\Big\vert\sum_{l\nin\served{i-1}\cup{\{j\}}}\! \FE \Big(\frac{S_{c(i)}S_l^{\alpha}}{n}\Big) - \frac{\lambda S_{c(i)}S_l^{\alpha}}{n}\Big\vert \mid \mathcal F_{i-1}\Big] \nnl
&\leq \sum_{l\in[n]}\ES\Big[\Big\vert \FE \Big(\frac{S_{c(i)}S_l^{\alpha}}{n}\Big)- \frac{\lambda S_{c(i)}S_l^{\alpha}}{n}\Big\vert \mid \mathcal F_{i-1}\Big]\nnl
&= \frac{1}{2n^2} \ES[\vert\FE''(\zeta_n)S_{c(i)}^2\vert \mid \mathcal F_{i-1}]\sum_{l\in[n]}S_l^{2\alpha}\leq \frac{\lambda^2}{2n^2}\ES[S_{c(i)}^2 \mid \mathcal F_{i-1}]\sum_{l\in[n]}S_l^{2\alpha},
\end{align}%
since $\vert \FE''(x)\vert \leq \lambda ^2$ for all $x\geq 0$. By Lemma \ref{lem:1+alphaConditionalMoment} this can be bounded by
\begin{align}\label{eq:F_taylor_expansion_error}%
\frac{\lambda^2}{2n^2}(C_n+ o_{\mathbb P}(1))\sum_{l\in[n]}S_l^{2\alpha},
\end{align}%
where $C_n$ is bounded w.h.p.~and the $o_{\mathbb P}(1)$ term is uniform in $i=O(n^{2/3})$. Therefore, the third term in \eqref{eq:conditioned_arrivals_first_eq} is $o_{\mathbb P}(n^{-1/3})$.
The remaining terms in \eqref{eq:conditioned_arrivals_first_eq} can be simplified as
\begin{align}\label{eq:ConditionedArrivalsFirst_U}%
\ES&[A^{\sss U}_n(i) \mid \mathcal F_{i-1}] -1= \sum_{j\nin\served{i-1}}\PS (c(i) = j \mid \mathcal F_{i-1})c_{n,\beta}\lambda S_j\frac{\sum_{l\in[n]} S_l^{\alpha}}{n}\nnl
&\quad -  \sum_{j\nin\served{i-1}}\PS(c(i) = j \mid \mathcal F_{i-1})\sum_{l\in\served{i-1}} \frac{c_{n,\beta}\lambda S_jS_l^{\alpha}}{n} \nnl
&\quad- c_{n,\beta}\lambda\sum_{j\nin\served{i-1}}\PS(c(i) = j\mid \mathcal F_{i-1})\frac{S_j^{1+\alpha}}{n} - 1 + o_{\mathbb P}(n^{-1/3})\nnl
&=\Big(c_{n,\beta}\lambda\frac{\sum_{l\in[n]} S_l^{\alpha}}{n}\mathbb E[S_{c(i)}\mid\mathcal F_{i-1}]-1\Big) - c_{n,\beta}\ES[S_{c(i)} \mid \mathcal F_{i-1}]\sum_{l\in\served{i-1}}\lambda\frac{S_l^{\alpha}}{n} \nnl
&\quad- c_{n,\beta}\frac{\lambda}{n}\ES[S_{c(i)}^{1+\alpha}\mid\mathcal F_{i-1}]+ o_{\mathbb P}(n^{-1/3}).
\end{align}%
For the first term of \eqref{eq:ConditionedArrivalsFirst_U},  using $\frac{c}{a-b}= \frac{c}{a}+ \frac{c}{a-b}\frac{b}{a}$, with $a = \sum_{l\in[n]}S_l^{\alpha}$ and $b = \sum_{l \in \served{i-1}}S_l^{\alpha}$,
\begin{align}\label{eq:ConditionedArrivalsFirstTerm}%
c_{n,\beta}\lambda&\frac{\sum_{l=1}^n S_l^{\alpha}}{n}\ES[S_{c(i)} \mid \mathcal F_{i-1}]-1  \nnl
&=c_{n,\beta}\lambda\frac{\sum_{l\in[n]} S_l^{\alpha}}{n} \sum_{j\nin \served{i-1}} \frac{S_j^{1+\alpha}}{\sum_{l\in [n]}S_l^{\alpha}}-1 + c_{n,\beta}\lambda\frac{\sum_{l\in[n]} S_l^{\alpha}}{n}\sum_{j\nin\served{i-1}}\frac{S_j^{1+\alpha}}{\sum_{l\nin\served{i-1}}S_l^{\alpha}}\frac{\sum_{s\in\served{i-1}}S_s^{\alpha}}{\sum_{l\in[n]}S_l^{\alpha}}\nnl
&= \Big(c_{n,\beta}\frac{\lambda}{n} \sum_{j\nin \served{i-1}} S_j^{1+\alpha} - 1 \Big)+ c_{n,\beta}\ES[S_{c(i)}\mid  \mathcal F_{i-1}]\sum_{s\in\served{i-1}}\lambda\frac{S_s^{\alpha}}{n}.
\end{align}%
Note that the right-most term in \eqref{eq:ConditionedArrivalsFirstTerm} and the second term in \eqref{eq:ConditionedArrivalsFirst_U} cancel out.  This cancellation is what makes the analysis of $N_n^{\sss U}(\cdot)$ considerably easier than the analysis of $N_n(\cdot)$.

Moreover, Lemma \ref{lem:1+alphaConditionalMoment} implies that the third term in \eqref{eq:ConditionedArrivalsFirst_U} is also $\oP(n^{-1/3})$.
\eqref{eq:conditioned_arrivals_first_eq} is then simplified to
\begin{align}\label{eq:conditioned_arrivals_final_expression}%
\ES[A^{\sss U}_n(i) \mid  \mathcal F_{i-1}] -1 &= c_{n,\beta}\frac{\lambda}{n} \sum_{j\nin \served{i-1}} S_j^{1+\alpha} - 1
+ o_{\mathbb P}(n^{-1/3})\nnl
&= \Big(c_{n,\beta}\frac{\lambda}{n} \sum_{j=1}^n S_j^{1+\alpha} - 1\Big) - c_{n,\beta}\frac{\lambda}{n} \sum_{j\in \served{i-1}} S_j^{1+\alpha}
+ o_{\mathbb P}(n^{-1/3})\nnl
&= \Big(c_{n,\beta}\frac{\lambda}{n} \sum_{j=1}^n S_j^{1+\alpha} - 1\Big) - c_{n,\beta}\frac{\lambda}{n} \sum_{j=1}^{i-1} S_{c(j)}^{1+\alpha}
+ o_{\mathbb P}(n^{-1/3}),
\end{align}%
and the $o_{\mathbb P}(n^{-1/3})$ term is uniform in $i=O(n^{2/3})$. We are now able to compute
\begin{align}%
n^{-1/3}C_n^{\sss U}(tn^{2/3}) &= n^{-1/3}\sum_{i=1}^{tn^{2/3}}(\ES[A_n^{\sss U}(i) \mid \mathcal F_{i-1}] -1) \nnl
&= tn^{1/3} \Big(c_{n,\beta}\frac{\lambda}{n}\sum_{j=1}^n S_j^{1+\alpha} - 1\Big) - c_{n,\beta}\frac{\lambda}{n^{4/3}} \sum_{i=1}^{tn^{2/3}} \sum_{j=1}^{i-1} S_{c(j)}^{1+\alpha} + o_{\mathbb P}(1).
\end{align}%
Note that, since $\E[(S^{1+\alpha})^{\frac{2+\alpha}{1+\alpha}}]<\infty$, by the Marcinkiewicz and Zygmund Theorem \cite[Theorem 2.5.8]{durrett2010probability}, if $\alpha\in(0,1]$,
\begin{equation}%
c_{n,\beta}\frac{\lambda}{n}\sum_{j=1}^n S_j^{1+\alpha} = c_{n,\beta}\lambda\E[S^{1+\alpha}] + \oP(n^{-\frac{1}{2+\alpha}})=1 + \beta n^{-1/3} + \oP(n^{-\frac{1}{2+\alpha}}).
\end{equation}%
For $\alpha = 0$, by a similar result \cite[Theorem 2.5.7]{durrett2010probability}, for all $\varepsilon>0$,
\begin{equation}%
\frac{1}{n}\sum_{j=1}^n S_j = \E[S] + \oP(n^{-1/2}\log(n)^{1/2+\varepsilon}).
\end{equation}%
In particular,
\begin{equation}%
tn^{1/3} \Big(c_{n,\beta}\frac{\lambda}{n}\sum_{j=1}^n S_j^{1+\alpha} - 1\Big) = t(\beta  + \oP(1)).
\end{equation}%
By monotonicity,
\begin{equation}%
\sup_{t\leq T}\Big\vert tn^{1/3} \Big(c_{n,\beta}\frac{\lambda}{n}\sum_{j=1}^n S_j^{1+\alpha} - 1\Big) - \beta t \Big\vert \Pconv 0,
\end{equation}%
so that, for $\alpha\in[0,1]$,
\begin{equation}\label{eq:drift_rescaled_final_expression}%
n^{-1/3}C_n^{\sss U}(tn^{2/3}) = \beta t  - c_{n,\beta}\frac{\lambda}{n^{4/3}} \sum_{i=1}^{tn^{2/3}} \sum_{j=1}^{i-1} S_{c(j)}^{1+\alpha} + o_{\mathbb P}(1).
\end{equation}%
Since $c_{n,\beta} = 1 + O(n^{-1/3})$, the second term in \eqref{eq:drift_rescaled_final_expression} converges uniformly to $-t^2\lambda\E[S^{1+2\alpha}]/2\E[S^{\alpha}]$ by Lemma \ref{lem:size_biased_service_times_unif_convergence}.

\subsubsection{Proof of \ref{size_biased:eq:second_condition_martingale_FCLT} for the upper bound}\label{sec:quadratic_variation_of_upper_bounding_process_converges_to_zero}

Rewrite $B_n^{\sss U}(k)$, for $k = O(n^{2/3})$, as
\begin{align}\label{eq:quadratic_variation_zeroth_U}%
B_n^{\sss U}(k) &= \sum_{i=1}^k(\ES[A_n^{\sss U}(i)^2 \mid \mathcal F_{i-1}] - \ES[A_n^{\sss U}(i)\vert \mathcal F_{i-1}]^2) \nnl
&= \sum_{i=1}^k(\ES[A_n^{\sss U}(i)^2  \mid \mathcal F_{i-1}] - 1) + \OP(kn^{-1/3}),
\end{align}%
where we have used the asymptotics for $\ES[A_n^{\sss U}(i) \mid  \mathcal F_{i-1}]$  in  \eqref{eq:conditioned_arrivals_final_expression}-\eqref{eq:drift_rescaled_final_expression}. Moreover, we can compute $\ES[A_n^{\sss U}(i)^2 \mid \mathcal F_{i-1}]$ as
\begin{align}\label{eq:quadratic_variation_first_U}%
\ES[A_n^{\sss U}(i)^2 \mid \mathcal F_{i-1}] &= \ES\Big[\Big(\sum_{h\nin\served{i}}\mathds{1}_{\{T_{h}\leq c_{n,\beta}S _{c(i)}S_h/n\}}\Big)^2 \mid \mathcal F_{i-1}\Big] \\
&= \ES[A_n^{\sss U}(i) \mid \mathcal F_{i-1}] + \ES\Big[\sum_{h,q\nin \served{i}}\mathds{1}_{\{T_{h}\leq  c_{n,\beta}S _{c(i)}S_h/n\}}\mathds{1}_{\{T_{q}\leq c_{n,\beta} S_{c(i)}S_q/n\}} \mid \mathcal F_{i-1}\Big].\notag
\end{align}%
Again by \eqref{eq:conditioned_arrivals_final_expression}, $\ES[A_n(i) \mid \mathcal F_{i-1}] = 1 + o_{\mathbb P}(1)$, uniformly in $i = O(n^{2/3})$, so that \eqref{eq:quadratic_variation_zeroth_U} simplifies to
\begin{equation}%
B_n(k) = \sum_{i=1}^k \ES\Big[\sum_{h,q\nin \served{i}}\mathds{1}_{\{T_{h}\leq  c_{n,\beta}S _{c(i)}S_h^{\alpha}/n\}}\mathds{1}_{\{T_{q}\leq  c_{n,\beta}S_{c(i)}S_q^{\alpha}/n\}} \mid \mathcal F_{i-1}\Big] + \OP(kn^{-1/3}).
\end{equation}%
We then focus on the second term in \eqref{eq:quadratic_variation_first_U}, which we compute as
\begin{align}\label{eq:BComputations_U}%
\sum_{\substack{h,q\nin \served{i}\\h\neq q}}&\ES[\mathds{1}_{\{T_{h}\leq  c_{n,\beta}S_{c(i)}S_h^{\alpha}/n\}}\mathds{1}_{\{T_{q}\leq  c_{n,\beta}S_{c(i)}S_q^{\alpha}/n\}} \mid \mathcal F_{i-1}]  \\
&= \sum_{j\nin\served{i-1}}\PS(c(i) = j \mid  \mathcal F_{i-1})\sum_{\substack{h,q\nin \served{i-1}\cup \{j\}\\h\neq q}}\ES[\mathds{1}_{\{T_{h}\leq  c_{n,\beta}S_{j}S_h^{\alpha}/n\}}\mathds{1}_{\{T_{q}\leq  c_{n,\beta}S_{j}S_q^{\alpha}/n\}} \mid \mathcal F_{i-1}].\nnl
\end{align}%
By Lemma \ref{lem:n_squared_err_term},
\begin{align}%
\textrm{l.h.s.}~\eqref{eq:BComputations_U}&= \sum_{j\nin\served{i-1}} \frac{S_j^{\alpha}}{\sum_{l\nin\served{i-1}}S_l^{\alpha}} \sum_{\substack{h,q\nin \served{i-1}\cup \{j\}\\h\neq q}}\Big(\frac{c_{n,\beta}^2\lambda^2 S_j^2S_h^{\alpha}S_q^{\alpha}}{n^2} + \oP(n^{-2})\Big)\nnl
&= (c_{n,\beta}\lambda) ^2\ES[S_{c(i)}^2  \mid \mathcal F_{i-1}]\frac{1}{n^{2}}\sum_{\substack{h,q\nin \served{i-1}\cup \{c(i)\}\\h\neq q}}S_h^{\alpha}S_q^{\alpha} + \oP(1)\nnl
&= \frac{(c_{n,\beta}\lambda) ^2}{n^2}\ES[S_{c(i)}^2 \mid \mathcal F_{i-1}]\!\sum_{\substack{1\leq h,q\leq n}}\!S_h^{\alpha}S_q^{\alpha} \nnl
&\qquad- \frac{(c_{n,\beta}\lambda) ^2}{n^2}\ES\Big[S_{c(i)}^2\mkern-9mu\sum_{\substack{h,q\in \served{i-1} \cup \{c(i)\}\\ \cup \{h = q\}}}\mkern-9mu S_h^{\alpha}S_q^{\alpha} \mid \mathcal F_{i-1}\Big] + o_{\mathbb P}(1).\notag
\end{align}%
The leading contribution to $B_n^{\sss U}(k)$ is given by the first term, while the second term is an error term by Lemma \ref{lem:unif_convergence_random_sums}. We have shown that $B_n^{\sss U}(\cdot)$ can be rewritten as
\begin{align}%
B_n^{\sss U}(k) = \Big(\frac{\lambda}{n}\sum_{h\in[n]}S_h^{\alpha}\Big)^2\sum_{i=1}^k  \ES[S_{c(i)}^2  \mid \mathcal F_{i-1}] + \oP(k).
\end{align}%
Thus,
\begin{equation}%
n^{-2/3}B_n^{\sss U}(n^{2/3}u) \stackrel{\mathbb P}{\rightarrow} \lambda^2\mathbb E[S^{\alpha}]\E [S^{2+\alpha}]u,
\end{equation}%
which concludes the proof of \ref{size_biased:eq:second_condition_martingale_FCLT}.
\qed

\subsubsection{Proof of \ref{size_biased:eq:third_condition_martingale_FCLT} for the upper bound}\label{sec:quadratic_variation_jumps_bounding_process_converges_to_zero}
The jumps of $B_n^{\sss U}(k)$ are given by
\begin{align}%
B_n^{\sss U}(i)-B_n^{\sss U}(i-1) &= \ES[A_n^{\sss U}(i)^2 \mid \mathcal F_{i-1}] - \ES[A_n^{\sss U}(i) \mid \mathcal F_{i-1}]^2\nnl
&=\ES\Big[\sum_{\substack{h,q\nin \served{i}\\h\neq q}}\mathds{1}_{\{T_{h}\leq  c_{n,\beta}S_{c(i)}S_h^{\alpha}/n\}}\mathds{1}_{\{T_{q}\leq  c_{n,\beta}S _{c(i)}S_q^{\alpha}/n\}} \mid \mathcal F_{i-1}\Big]\nnl
&\qquad+ \left(\ES[A_n^{\sss U}(i) \mid \mathcal F _{i-1}] - \ES[A_n^{\sss U}(i) \mid \mathcal F_{i-1}]^2\right)
\end{align}%
Since $\ES[A_n^{\sss U}(i) \mid \mathcal F _{i-1}] = 1 + \OP(n^{-1/3})$ for $i= \OP(n^{2/3})$ by  \eqref{eq:conditioned_arrivals_final_expression}, the second term is of order $O_{\mathbb P}(n^{-1/3})$,  uniformly in $i = \OP(n^{2/3})$. The first term was computed in \eqref{eq:BComputations_U}. Therefore,
\begin{align}%
&B_n^{\sss U}(i) - B_n^{\sss U}(i-1) \nnl
&\quad = \frac{(c_{n,\beta}\lambda) ^2}{n^{2}}\ES[S_{c(i)}^2 \mid \mathcal F_{i-1}]\sum_{h,q\in[n]}S_h^{\alpha}S_q^{\alpha} - \frac{(c_{n,\beta}\lambda) ^2}{n^2}\ES\Big[S_{c(i)}^2\mkern-9mu\sum_{\substack{h,q\in \served{i-1}\cup \{c(i)\}\\ \cup \{h = q\}}}\mkern-9mu S_h^{\alpha}S_q^{\alpha} \mid \mathcal F_{i-1}\Big] + \oP(1)\nnl
&\quad \leq \frac{(c_{n,\beta}\lambda) ^2}{n^{2}}\ES[S_{c(i)}^2 \mid \mathcal F_{i-1}]\sum_{h,q\in [n]}S_h^{\alpha}S_q^{\alpha}.
\end{align}%
After rescaling and taking the expectation, we obtain the bound
\begin{equation}\label{eq:proof_of_iii_final_bound_U}
n^{-2/3}\ES[\sup_{i\leq \bar t n^{2/3}}\vert B_n^{\sss U}(i) - B_n^{\sss U}(i-1)\vert]\leq \frac{(c_{n,\beta}\lambda)^2}{n^{2/3}}\ES[\sup_{i\leq \bar t n^{2/3}}S_{c(i)}^2] \Big(\frac{1}{n}\sum_{h,q\in[n]} S_h^{\alpha} \Big)^2.
\end{equation}%
\begin{lemma}\label{lem:exp_sup_S_squared_oP_n_two_thirds}%
If $\E[S^{2+\alpha}]<\infty$,
\begin{equation}%
\ES[\sup_{k\leq tn^{2/3}}S_{c(k)}^2] = \oP(n^{2/3}).
\end{equation}%
\end{lemma}%
\begin{proof}
For $\varepsilon>0$ split the expectation as
\begin{align}\label{eq:first_way_beginning}%
\ES [(\sup_{k\leq tn^{2/3}}S_{c(k)})^2]\leq \ES [\sup_{k\leq tn^{2/3}}S_{c(k)}^2\mathds 1_{\{S_{c(k)}> \varepsilon n^{1/3}\}}] + \varepsilon^2 n^{2/3}.
\end{align}%
We bound the expected value in the first term as
\begin{align}%
\ES [\sup_{k\leq tn^{2/3}}S^2_{c(k)} \mathds 1_{\{S_{c(k)}> \varepsilon n^{1/3}\}}] &\leq \sum_{k\leq tn^{2/3}} \frac{1}{n^{2/3}}\ES[S_{c(k)}^2\mathds 1_{\{S_{c(k)}> \varepsilon n^{1/3}\}}]\nnl
&\leq n^{2/3} t  \ES[S_{c(1)}^2 \mathds 1_{\{S_{c(1)}> \varepsilon n^{1/3}\}}](1 + \OPS(1)),
\end{align}%
where we used  Lemma \ref{lem:expectation_service_times_ordering} with $f(x) = x^2\mathds 1_{\{x>\varepsilon n^{1/3}\}}$.
Computing the expectation explicitly we get
\begin{align}%
t \ES[S_{c(1)}^2 \mathds 1_{\{S_{c(1)}> \varepsilon n^{1/3}\}}] &= t \sum_{i\in[n]}S_i^2 \mathds 1_{\{S_{i}> \varepsilon n^{1/3}\}} \mathbb{P}( c(1) = i) \nnl
&= t \sum_{i\in[n]}S_i^2 \mathds 1_{\{S_{i}> \varepsilon n^{1/3}\}}\frac{S_i^{\alpha}}{\sum_{j\in[n]} S_j^{\alpha}},
\end{align}%
so that the left-hand side of \eqref{eq:first_way_beginning} is bounded by
\begin{align}%
 \frac{t}{\sum_{j\in[n]} S_j^{\alpha}} \sum_{i\in[n]}  S_i^{2+\alpha}\mathds 1_{\{S_{i}> \varepsilon n^{1/3}\}} + \Big(\sum_{i\in[n]}\frac{S_i^{\alpha}}{n}\Big)^2\varepsilon^2,
\end{align}%
which tends to zero as $n\rightarrow\infty$ since $\mathbb E[S^{2+\alpha}] < \infty$ and $\varepsilon>0$ is arbitrary.
\end{proof}

By Lemma \ref{lem:exp_sup_S_squared_oP_n_two_thirds} the right-hand side of \eqref{eq:proof_of_iii_final_bound_U} converges to zero, and this concludes the proof.
\qed

\subsubsection{Proof of \ref{size_biased:eq:fourth_condition_martingale_FCLT} for the upper bound}\label{sec:martingale_jumps_bounding_process_converges_to_zero}
First we split
\begin{align}%
\ES[\sup_{k\leq tn^{2/3}}(M_n^{\sss U}(k) - M_n^{\sss U}(k-1))^2] &= \ES[\sup_{k\leq tn^{2/3}}(A_n^{\sss U}(k) - \ES[A_n^{\sss U}(k) \mid \mathcal F_{k-1}])^2]\nnl
&\leq \ES[ \sup_{k\leq tn^{2/3}}\vert A_n^{\sss U}(k)\vert^2] + \ES[\sup_{k\leq tn^{2/3}}\mathbb E[A_n^{\sss U}(k) \mid \mathcal F_{k-1}]^2]\nnl
&\leq 2\ES [ \sup_{k\leq tn^{2/3}}\vert A_n^{\sss U}(k)\vert^2].
\end{align}%
We then stochastically dominate $(A_n^{\sss U}(k))_{k\leq tn^{2/3}}$ by a sequence of  Poisson processes $(\Pi_k)_{k\leq tn^{2/3}}$, according to
\begin{equation}\label{eq:expectation_sup_A_squared_start}%
A_n^{\sss U}(k) \preceq \Pi_k\Big(c_{n,\beta} S_{c(k)}\sum_{i\in[n]} \frac{S_i^{\alpha}}{n}\Big)=:A'_n(k).
\end{equation}%
Indeed, if $E_1, E_2,\ldots, E_n$ are exponential random variables with parameters $\lambda_1, \lambda_2, \ldots, \lambda_n$, there exists a coupling with a Poisson process $\Pi(\cdot)$ such that $\sum_{i\leq n} \mathds 1_{\{E_i\leq t\}} \leq \Pi(\sum_{i\leq n}\lambda_i t)$. The coupling is constructed as follows. Each random variable $E_i$ is coupled with a Poisson process $\Pi^{(i)}$ with intensity $\lambda_i$ in such a way that $\mathds 1_{\{E_i\leq t\}} \leq \Pi^{(i)} (\lambda_i t)$. Moreover, by basic properties of the Poisson process $\sum_{i\leq n} \Pi^{(i)}(\lambda_i t) \sr{\mathrm d}{=} \Pi(\sum_{i\leq n} \lambda_i t)$.

We bound \eqref{eq:expectation_sup_A_squared_start} via martingale techniques. First, we decompose it as
\begin{align}\label{eq:expectation_sup_A_squared_martingale_decomposition}%
n^{-2/3}\ES [ \sup_{k\leq tn^{2/3}}\vert A^{\sss U}_n(k)\vert^2] \leq &2n^{-2/3}\ES \Big[\Big(\sup_{k\leq tn^{2/3}} \Big\vert A'_n(k) - c_{n,\beta}S_{c(k)}\sum_{i\in[n]}\frac{S_i^{\alpha}}{n}\Big\vert\Big)^2\Big] \nnl
&+ 2n^{-2/3}\ES \Big[\Big(c_{n,\beta}\sup_{k\leq tn^{2/3}} S_{c(k)}\sum_{i\in[n]} \frac{S_i^{\alpha}}{n}\Big)^2\Big]
\end{align}%
Applying Doob's $L^2$ martingale inequality  \cite[Theorem 11.2]{klenke2008probability} to the first term we see that it converges to zero, since
\begin{align}%
n^{-2/3}\ES \Big[\Big(\sup_{k\leq tn^{2/3}}\vert A'_n(k) - S_{c(k)}\sum_{i\in[n]} \frac{S_i^{\alpha}}{n}\Big\vert\Big)^2\Big] &\leq 4n^{-2/3}\ES\Big[\Big\vert A'_n(tn^{2/3}) - S_{c(tn^{2/3})}\sum_{i\in[n]}\frac{S_i^{\alpha}}{n}\Big\vert^2\Big]\nnl
&= 4n^{-2/3} \ES \Big[S_{c(tn^{2/3})}\sum_{i\in[n]}\frac{S_i^{\alpha}}{n}\Big].
\end{align}%
The last equality follows from the expression for the variance of a Poisson random variable. The right-most term converges to zero by Lemma \ref{lem:1+alphaConditionalMoment}. We now bound the second term in \eqref{eq:expectation_sup_A_squared_martingale_decomposition}, as
\begin{align}\label{eq:expectation_sup_A_squared_martingale_drift_estimate}%
n^{-2/3}\ES \Big[\Big(\sup_{k\leq tn^{2/3}} S_{c(k)}\sum_{i\in[n]} \frac{S_i^{\alpha}}{n}\Big)^2\Big] &= n^{-2/3}\Big(\sum_{i\in[n]} \frac{S_i^{\alpha}}{n}\Big)^2  \ES [(\sup_{k\leq tn^{2/3}}S_{c(k)})^2]
%
\end{align}%
By Lemma \ref{lem:exp_sup_S_squared_oP_n_two_thirds} the right-hand side of \eqref{eq:expectation_sup_A_squared_martingale_drift_estimate} converges to zero, concluding the proof of \ref{size_biased:eq:fourth_condition_martingale_FCLT}.
\qed

\subsection{Convergence of the scaling limit}
As a consequence of \eqref{eq:reflected_process_domination} and Theorem \ref{MainTheorem_U} we have  that $Q_n(k) = \OP(n^{1/3})$ for $k = O(n^{2/3})$. In fact, $n^{-1/3}Q_n(k)$ is tight when $k=O(n^{2/3})$, as the following lemma shows:

\begin{lemma}\label{lem:sup_Q_converges_to_zero}%
Fix $\bar t>0$. The sequence $n^{-1/3}\sup_{t\leq \bar t} Q_n(t n^{2/3})$ is tight.

\end{lemma}%

\begin{proof}
The supremum function $f(\cdot)\mapsto \sup_{t\leq \bar t}f(t)$ is continuous in $(\mathcal D, J_1)$ by \cite[Theorem 13.4.1]{StochasticProcess}. In particular,
\begin{equation}%
n^{-1/3}\sup_{t\leq \bar t} Q_n^{\sss U}(t n^{2/3})\dconv\sup_{t\leq \bar t} W(t),\qquad\mathrm{in}~(\mathcal D, J_1).
\end{equation}%
Since $Q_n(k) \leq Q_n^{\sss U}(k)$, the conclusion follows.
\end{proof}%
As an immediate consequence of \eqref{eq:cardinality_nu} and Lemma \ref{lem:sup_Q_converges_to_zero}, we have the following important corollary. Recall that  $\nu_i$ is the set of customers who have left the system or are in the queue at the beginning of the $i$-th service, so that $\vert \nu_i\vert = i + Q_n(i)$. Recall also that $0\leq Q_n(t) \leq Q_n^{\sss U}(t)$.
\begin{corollary}\label{cor:nr_joined_customer_at_time_n_twothirds_is_n_twothirds}
As $n\rightarrow\infty$,
\begin{equation}%
\vert \nu_i \vert = i + \oP (i), \qquad \mathrm{uniformly~in}~i= \OP (n^{2/3}).
\end{equation}%
\end{corollary}
Intuitively, this implies that the main contribution to the downwards drift in the queue-length process comes from the customers that have left the system, and not from the customers in the queue. Alternatively, the order of magnitude of the queue length, that is $n^{1/3}$, is negligible with respect to the order of magnitude of the customers who have left the system, which is $n^{2/3}$.

In order to prove Theorem \ref{th:main_theorem_delta_G_1} we proceed as in the proof of Theorem \ref{MainTheorem_U}, but we now need to deal with the more complicated drift term. As before, we decompose $N_n(k) = M_n(k) + C_n(k)$, where
\begin{align}%
M_n(k) & = \sum_{i=1}^k (A_n(i) - \ES[A_n(i) \mid \mathcal F_{k-1}]),\nnl
C_n(k) &= \sum_{i=1}^k (\ES[A_n(i) \mid \mathcal F_{k-1}] -1),\nnl
B_n(k) &= \sum_{i=1}^k (\ES[A_n(i)^2 \mid \mathcal F_{i-1}] - \ES[A_n(i) \mid \mathcal F_{i-1}]^2).
\end{align}%
As before, we separately prove the convergence of the drift $C_n(k)$ and of the martingale $M_n(k)$, by verifying the conditions \ref{size_biased:eq:first_condition_martingale_FCLT}-\ref{size_biased:eq:fourth_condition_martingale_FCLT} in Section \ref{sec:proof_thm_U}.
Verifying \ref{size_biased:eq:first_condition_martingale_FCLT} proves to be the most challenging task, while the estimates for \ref{size_biased:eq:second_condition_martingale_FCLT}-\ref{size_biased:eq:fourth_condition_martingale_FCLT} in Section \ref{sec:proof_thm_U} carry over without further complications.
\subsubsection{Proof of \ref{size_biased:eq:first_condition_martingale_FCLT} for the embedded queue}
By expanding $\ES[A_n(i) \mid \mathcal F_{i-1}] -1$ as in \eqref{eq:ConditionedArrivalsFirst_U}, we get
\begin{align}\label{eq:ConditionedArrivalsFirst}%
\ES[A_n(i) \mid \mathcal F_{i-1}] -1&= \Big(c_{n,\beta}\lambda\frac{\sum_{l=1}^n S_l^{\alpha}}{n}\ES[S_{c(i)} \mid \mathcal F_{i-1}]-1\Big) - c_{n,\beta}\ES[S_{c(i)} \mid \mathcal F_{i-1}]\sum_{l\in\nu_{i}\setminus \{c(i)\}}\lambda\frac{S_l^{\alpha}}{n} \notag\\
&\quad- c_{n,\beta}\frac{\lambda}{n}\ES[S_{c(i)}^{1+\alpha} \mid \mathcal F_{i-1}]+ o_{\mathbb P}(n^{-1/3}).
\end{align}%
By further expanding the first term in \eqref{eq:ConditionedArrivalsFirst} as in \eqref{eq:ConditionedArrivalsFirstTerm}, we get
\begin{align}\label{eq:conditioned_arrivals_middle_computation}%
\ES[A_n(i) \mid \mathcal F_{i-1}] - 1 &= \Big(c_{n,\beta} \frac{\lambda}{n}\sum_{j\nin\served{i-1}}S_j^{1+\alpha}-1\Big) - c_{n,\beta}\ES[S_{c(i)} \mid \mathcal F_{i-1}] \sum_{l=i+1}^{i+1+Q_n(i-1)} \lambda \frac{S_{c(l)}^{\alpha}}{n} \nnl
&\quad- c_{n,\beta}\frac{\lambda}{n}\ES[S_{c(i)}^{1+\alpha} \mid \mathcal F_{i-1}] + \oP(n^{-1/3}),
\end{align}%
where in the first equality we have used \eqref{eq:cardinality_nu}. Comparing equation \eqref{eq:conditioned_arrivals_middle_computation} with equation \eqref{eq:conditioned_arrivals_final_expression}, we rewrite the drift as
\begin{equation}%
C_n(k) = C_n^{\sss U}(k) - c_{n,\beta}\lambda\sum_{i=1}^{k}\ES[S_{c(i)} \mid \mathcal F_{i-1}] \sum_{l=i+1}^{i+1+Q_n(i-1)}  \frac{S_{c(l)}^{\alpha}}{n}.
\end{equation}%
Therefore, to conclude the proof of \ref{size_biased:eq:first_condition_martingale_FCLT} it is enough to show that the second term vanishes, after rescaling. We do this in the following lemma:
\begin{lemma}\label{lem:difficult_drift_term_converges_to_zero}
As $n\rightarrow\infty$,
\begin{equation}%
n^{-1/3}c_{n,\beta}\lambda\sum_{i=1}^{\bar t n^{2/3}}\ES[S_{c(i)} \mid \mathcal F_{i-1}] \sum_{l=i+1}^{i+1+Q_n(i-1)}  \frac{S_{c(l)}^{\alpha}}{n} \Pconv 0.
\end{equation}%
\end{lemma}
\begin{proof}%
By Lemma \ref{lem:sup_Q_converges_to_zero}, $\sup_{i\leq \bar t n^{2/3}}Q_n(i)\leq C_1 n^{1/3}$ w.h.p.~for a large constant $C_1$, and by Lemma \ref{lem:1+alphaConditionalMoment}, $\sup_{i\leq \bar tn^{2/3}}\ES[S_{c(i)} \mid \mathcal F_{i-1}]\leq C_2$ w.h.p.~for another large constant $C_2$. This implies that, w.h.p.,
\begin{equation}%
n^{-1/3}c_{n,\beta}\lambda\sum_{i=1}^{\bar t n^{2/3}}\ES[S_{c(i)} \mid \mathcal F_{i-1}] \sum_{l=i+1}^{i+1+Q_n(i-1)}  \frac{S_{c(l)}^{\alpha}}{n}  \leq c_{n,\beta}\lambda C_2 \sum_{i=1}^{\bar t n^{2/3}}\sum_{l=i+1}^{i+1+C_1n^{1/3}}\frac{S_{c(l)}^{\alpha}}{n^{4/3}}.
\end{equation}%
The double sum can be rewritten as
\begin{align}%
c_{n,\beta}\lambda C_2 \sum_{i=1}^{\bar t n^{2/3}}\sum_{l=i+1}^{i+1+C_1n^{1/3}}\frac{S_{c(l)}^{\alpha}}{n^{4/3}} &\leq c_{n,\beta}\lambda C_2 \sum_{j=1}^{\bar tn^{2/3}+C_1n^{1/3}}\min\{j, C_1 n^{1/3}\}\frac{S^{\alpha}_{c(j)}}{n^{4/3}}\nnl
&\leq c_{n,\beta}\lambda C_1 C_2 \sum_{j=1}^{(\bar t+C_1)n^{2/3}} \frac{S^{\alpha}_{c(j)}}{n}.
\end{align}%
The right-most term converges to zero in probability as $n\rightarrow\infty$ by Lemma \ref{lem:size_biased_service_times_unif_convergence}. This concludes the proof.
\end{proof}%
Since
\begin{equation}%
n^{-1/3}C_n(tn^{2/3}) = n^{-1/3}C_n^{\sss U}(tn^{2/3}) - n^{-1/3}c_{n,\beta}\lambda\sum_{i=1}^{tn^{2/3}}\ES[S_{c(i)} \mid \mathcal F_{i-1}] \sum_{l=i+1}^{i+1+Q_n(i-1)}  \frac{S_{c(l)}^{\alpha}}{n},
\end{equation}%
Lemma \ref{lem:difficult_drift_term_converges_to_zero} and the convergence result \eqref{eq:drift_rescaled_final_expression} for $ n^{-1/3}C_n^{\sss U}(tn^{2/3})$ conclude the proof of \ref{size_biased:eq:first_condition_martingale_FCLT}.
\qed

\subsubsection{Proof of \ref{size_biased:eq:second_condition_martingale_FCLT}, \ref{size_biased:eq:third_condition_martingale_FCLT} and \ref{size_biased:eq:fourth_condition_martingale_FCLT} for the embedded queue}
Proceeding as before, we find that
%
%
\begin{align}\label{eq:quadratic_variation_zeroth}%
B_n(k) &= \sum_{i=1}^k(\ES[A_n(i)^2 \mid \mathcal F_{i-1}] - \ES[A_n(i) \mid \mathcal F_{i-1}]^2) \nnl
&= \sum_{i=1}^k(\ES[A_n(i)^2 \mid \mathcal F_{i-1}] - 1) + \OP(kn^{-1/3}),
\end{align}%
%
%
where
\begin{equation}\label{eq:quadratic_variation_first}%
\ES[A_n(i)^2 \mid \mathcal F_{i-1}] = \ES[A_n(i)\mid \mathcal F_{i-1}] + \ES\Big[\sum_{\substack{h,q\nin \nu_{i-1}\\h\neq q}}\mathds{1}_{\{T_{h}\leq  S _{c(i)}S_h/n\}}\mathds{1}_{\{T_{q}\leq  S_{c(i)}S_q/n\}} \mid \mathcal F_{i-1}\Big].
\end{equation}%
%
Similarly as in Section \ref{sec:quadratic_variation_of_upper_bounding_process_converges_to_zero}, we get
\begin{align}\label{eq:BComputations}%
\sum_{\substack{h,q\nin \nu_{i-1}\\h\neq q}}&\ES[\mathds{1}_{\{T_{h}\leq  S_{c(i)}S_h^{\alpha}/n\}}\mathds{1}_{\{T_{q}\leq  S_{c(i)}S_q^{\alpha}/n\}} \mid \mathcal F_{i-1}]  \\
&= \ES[S_{c(i)}^2 \mid \mathcal F_{i-1}]\frac{\lambda ^2}{n^{2}}\Big(\sum_{h=1}^nS_h^{\alpha}\Big)^2 - \ES\Big[S_{c(i)}^2\frac{\lambda ^2}{n^{2}}\sum_{\substack{h,q\in \nu_{i-1}\cup \{c(i)\}\\ \cup \{h = q\}}}S_h^{\alpha}S_q^{\alpha} \mid \mathcal F_{i-1}\Big] + \oP(1)\notag.
\end{align}%
The second term is an error term by Lemma \ref{lem:unif_convergence_random_sums} and Corollary \ref{cor:nr_joined_customer_at_time_n_twothirds_is_n_twothirds}. This implies that $B_n(\cdot)$ can be rewritten as
\begin{equation}%
B_n(k) = \Big(\frac{\lambda}{n}\sum_{h=1}^nS_h^{\alpha}\Big)^2\sum_{i=1}^k  \ES[S_{c(i)}^2 \mid \mathcal F_{i-1}] + o_{\mathbb P}(k),
\end{equation}%
so that
\begin{equation}%
n^{-2/3}B_n(n^{2/3}u) \stackrel{\mathbb P}{\rightarrow} \lambda^2\mathbb E[S^{\alpha}]\E [S^{2+\alpha}]u,
\end{equation}%
which concludes the proof of \ref{size_biased:eq:second_condition_martingale_FCLT}.
\qed

To conclude the proof of Theorem \ref{th:main_theorem_delta_G_1}, we are left to verify \ref{size_biased:eq:third_condition_martingale_FCLT} and \ref{size_biased:eq:fourth_condition_martingale_FCLT}. However, the estimates in Sections \ref{sec:quadratic_variation_jumps_bounding_process_converges_to_zero} and  \ref{sec:martingale_jumps_bounding_process_converges_to_zero} also hold for $B_n(\cdot)$ and $M_n(\cdot)$, since they rely respectively on \eqref{eq:proof_of_iii_final_bound_U} and \eqref{eq:expectation_sup_A_squared_start} to bound the lower-order contributions to the drift. This concludes the proof of Theorem \ref{th:main_theorem_delta_G_1}.
\qed

\section{Conclusions and discussion}\label{sec:conclusion}
{In this paper we have considered a generalization of the $\DG$ queue, which we coined the $\DGa$ queue, a model for the dynamics of a queueing system in which only a finite number of customers can join. In our model, the arrival time of a customer depends on its service requirement through a parameter $\alpha\in[0,1]$. We have proved that, under a suitable heavy-traffic assumption, the diffusion-scaled queue-length process embedded at service completions converges to a stochastic process $W(\cdot)$.
A distinctive characteristic of our results is the so-called \emph{depletion-of-points effect}, represented by a quadratic drift in $W(\cdot)$. A (directed) tree is associated to the $\DGa$ queue in a natural way, and the heavy-traffic assumption corresponds to \emph{criticality} of the associated random tree.
Our result interpolates between two already known results. For $\alpha = 0$ the arrival clocks are i.i.d.~and the analysis simplifies significantly. In this context, \cite{bet2014heavy} proves an analogous heavy-traffic diffusion approximation result. Theorem \ref{th:main_theorem_delta_G_1} can then be seen as a generalization of \cite[Theorem 5]{bet2014heavy}. If $\alpha = 1$, the $\DGa$ queue has a natural interpretation as an exploration process of an inhomogeneous random graph. In this context, \cite{bhamidi2010scaling} proves that the ordered component sizes converge to the excursion of a reflected Brownian motion with parabolic drift. Our result can then also be seen as a generalization of \cite{bhamidi2010scaling} to the \emph{directed} components of directed inhomogeneous random graphs.

Lemma \ref{lem:size_biased_distribution} implies that the distribution of the service time of the first $O(n^{2/3})$ customers to join the queue converges to the $\alpha$-\emph{size-biased} distribution of $S$, irrespectively of the precise time at which the customers arrive. This suggests that it is possible to prove Theorem \ref{th:main_theorem_delta_G_1} by approximating the $\DGa$ queue via a $\DG$ queue with service time distribution $S^*$ such that
\begin{equation}%
\mathbb P(S^*\in \mathcal A) = \E[S^{\alpha}\mathds 1_{\{S\in\mathcal A\}}]/\E[S^{\alpha}],
\end{equation}%
and i.i.d.~arrival times distributed as $T_i\sim\exp (\lambda\E[S^{\alpha}])$. This conjecture is supported by two observations. First, the heavy-traffic conditions for the two queues coincide. Second, the standard deviation of the Brownian motion is the same in the two limiting diffusions. However, this approximation fails to capture the higher-order contributions to the queue-length process. As a result, the coefficients of the negative quadratic drift in the two queues are different, and thus the approximation of the $\DGa$ queue with a $\DG$ queue is insufficient to prove Theorem \ref{th:main_theorem_delta_G_1}.

Surprisingly, the assumption that $\alpha$ lies in the interval $[0,1]$ plays no role in our proof. On the other hand, we see from \eqref{eq:main_theorem_delta_G_1_diffusion_definition} that
\begin{equation}\label{eq:moments_assumption_for_general_alpha}%
\max\{\E[S^{2+\alpha}],\E[S^{1+2\alpha}],\E[S^{\alpha}]\}<\infty
\end{equation}%
is a necessary condition for Theorem \ref{th:main_theorem_delta_G_1} to hold. From this we conclude that Theorem \ref{th:main_theorem_delta_G_1}  remains true as long as $\alpha\in\mathbb R$ is such that \eqref{eq:moments_assumption_for_general_alpha} is satisfied. From the modelling point of view, $\alpha>1$ represents a situation in which customers with larger job sizes have a stronger incentive to join the queue. On the other hand, when $\alpha < 0$ the queue models  a situation in which customers with large job sizes are lazy and thus favour joining the queue later. We remark that the \emph{form} of the limiting diffusion is the same for all $\alpha\in\mathbb R$, but different values of $\alpha$ yield different fluctuations (standard deviation of the Brownian motion), and a different quadratic drift.

\paragraph{Acknowledgments}

This work is supported by the NWO Gravitation {\sc Networks} grant 024.002.003. The work of RvdH is further supported by the NWO VICI grant 639.033.806. The work of GB and JvL is further supported by an NWO TOP-GO grant and by an ERC Starting Grant.

\bibliographystyle{abbrv}

\DeclareRobustCommand{\HOF}[3]{#3} 
\bibliography{library}

\end{document}